\documentclass[a4paper,10pt]{article}

\usepackage{amsmath, amsthm, amssymb,  amsfonts, amscd, color, tikz}

\definecolor{extlinkz}{rgb}{0,0,1}
\definecolor{intlinkz}{rgb}{1,0,0}
\definecolor{citlinkz}{rgb}{0,.6,.1}
\usepackage[pdfpagelabels,plainpages=false,colorlinks=true,citecolor=citlinkz,linkcolor=intlinkz,urlcolor=extlinkz,urlbordercolor={0 0 0}]{hyperref}

\allowdisplaybreaks

\newtheorem{theo}{Theorem}[section]

\newtheorem{lem}[theo]{Lemma}
\newtheorem{prop}[theo]{Proposition}
\newtheorem{cor}[theo]{Corollary}
\newtheorem{rem}[theo]{Remark}

\newtheorem{TheoA}{Theorem A}
\newtheorem{TheoB}{Theorem B}
\newtheorem{DEF}{Definition}

\newcommand{\R}{\mathbb{R}}
\newcommand{\C}{\mathbb{C}}
\newcommand{\Z}{\mathbb{Z}}
\newcommand{\T}{{\mathbb{T}}}
\newcommand{\Heis}{\mathbb{H}}

\newcommand{\Rn}{{\mathbb{R}^n}}
\newcommand{\Zn}{{\mathbb{Z}^n}}
\newcommand{\Tn}{{\mathbb{T}^n}}

\newcommand{\M}{\mathcal{M}}
\newcommand{\meas}{\mathcal{X}}

\newcommand{\G}{\mathbb{G}}
\newcommand{\vn}{\textnormal{vNa}}
\newcommand{\B}{\mathcal{B}}

\newcommand{\Hil}{\mathcal{H}}
\newcommand{\HD}{\ell_2(\Gamma)}

\newcommand{\id}{{\mathrm{e}}}
\newcommand{\Id}{\mathbf{1}}
\newcommand{\Proj}{\mathbf{P}}

\newcommand{\F}{\mathcal F}
\newcommand{\lr}{\lambda}
\newcommand{\act}{\sigma}
\newcommand{\ract}{\,\,{}_{\act}\!\!\rtimes}

\newcommand{\wh}[1]{\widehat{#1}}
\newcommand{\ol}[1]{\overline{#1}}

\newcommand{\Ran}{\mathcal{R}}
\newcommand{\Ker}{\textnormal{Ker}}
\newcommand{\vsp}{\textnormal{span}}

\newcommand{\psis}[1]{\langle #1\rangle}

\newcommand{\Frame}{\mathfrak{F}}

\newcommand{\ind}{\mathcal{I}}

\title{Riesz and frame systems generated by\\ unitary actions of discrete groups}
\author{D. Barbieri, E. Hern\'andez, J. Parcet}

\begin{document}

\maketitle

\begin{abstract}

We characterize orthonormal bases, Riesz bases and frames which arise from the action of a countable discrete group $\Gamma$ on a single element $\psi$ of a given Hilbert space $\Hil$. As $\Gamma$ might not be abelian, this is done in terms of a bracket map taking values in the $L^1$-space associated to the group von Neumann algebra of $\Gamma$. Our result generalizes recent work for LCA groups in \cite{HSWW10}. In many cases, the bracket map can be computed in terms of a noncommutative form of the Zak transform.
\let\thefootnote\relax\footnote{\hspace{-15pt}\textit{2010 Mathematics Subject Classification}: 42C15, 47A15, 43A65, 46L52 }

\end{abstract}

{
\hypersetup{linkcolor=black}
\tableofcontents
}

\section{Introduction}

Shift invariant subspaces are closed subspaces of the Hilbert space $L^2(\R^d)$ that are invariant under translations by elements of $(\Z^d , +)$. They have been the main subject of an increasing number of works in the last twenty five years due to their role in approximation theory \cite{Mallat89, Meyer93, BoorVoreRon93, BoorVoreRon94, RonShen95, BenedettoWalnut94, BenedettoLi98, Bownik00, HLWW02, RonShen05, HSWW10, CabrelliPaternostro10, BHM13}. Shift invariant spaces generated by the system of translations of a single function are called principal and play a fundamental role. Translations in $L^2(\R^d)$ by elements of $\mathbb{Z}^d$ are just an example of many other transformations given by the action of a discrete group $\Gamma$ in a Hilbert space $\mathcal{H}$. Translations, modulations or dilations in Euclidean spaces are standard actions in the construction of wavelets, Gabor frames and multiresolution analysis. All these actions arise from abelian groups and only the combination of them might result in a noncommutative action. The aim of this paper is to analyze the properties of principal shift invariant spaces in a Hilbert space $\mathcal{H}$ given by the action of an arbitrary countable discrete group $\Gamma$. To be precise, let
$$
\gamma \in \Gamma \mapsto T(\gamma) \in U(\Hil)
$$
be a unitary representation of $\Gamma$ in $\Hil$, and for $\psi \in \Hil$ set
$$
\psis{\psi} = \ol{\vsp\left\{T(\gamma) \psi \right\}_{\gamma \in \Gamma}}^{\Hil}.
$$
When does $\left\{T(\gamma) \psi \right\}_{\gamma \in \Gamma}$ become an orthonormal basis of $\psis{\psi}$? Similarly, when is it a Riesz basis or just a frame of the corresponding invariant space $\psis{\psi}$? In this paper, we will answer these questions in terms of the properties of an operator which takes values in the $L^1$-space associated to the group von Neumann algebra of $\Gamma$. This problem was already addressed for integer translations in $\R^d$ -- see \cite{Mallat89, Meyer93, Daubechies92} for orthonormal and Riesz bases and \cite{BenedettoWalnut94, BenedettoLi98} for frames -- and more general actions of locally compact abelian (LCA) groups studied in \cite{HSWW10}. Our solution follows the spirit of the bracket map introduced in \cite{HSWW10} and defined as follows. 

Let $\G$ be an LCA group, let $\wh{\G}$ be its Pontryagin dual and denote by $d\alpha$ the Haar measure on $\wh{\G}$. A unitary representation $T: \G \to U(\mathcal{H})$ is said to be dual
integrable if there exists a sesquilinear map
$$
[\cdot , \cdot] : \Hil \times \Hil \rightarrow L^1(\wh{\G}, d\alpha)
$$
such that
$$
\big\langle \varphi, T(g) \psi \big\rangle_{\Hil} = \int_{\widehat{\G}} [\varphi, \psi](\alpha) \overline{\alpha(g)} d\alpha \quad \forall \ \varphi, \psi \in \Hil .
$$

\noindent Let us illustrate this definition with two basic examples:
\begin{itemize}
\item In the case of integer translations $T_k f(x) = f(x + k)$ -- which yields a unitary representation of $(\Z^d,+)$ in $L^2(\R^d)$ -- the bracket map corresponds to the following periodization map
\begin{equation}\label{eq:integertranslates}
[\varphi, \psi](\xi) = \sum_{l \in \Z^d} \wh{\varphi}(\xi + l)\overline{\wh{\psi}(\xi + l)}
\end{equation}
where $\wh{\psi}$ is the Fourier transform of $\psi$, $\xi \in \T^d \simeq [0,1)^d$ and $\varphi, \psi \in L^2(\R^d)$ (see \cite{BoorVoreRon93, BoorVoreRon94, BenedettoLi98}).

\item It also includes other cases, such as the Gabor representation of the group $(\Z^d \times \Z^d,+)$ on $L^2(\R^d)$ given by $M_l T_k \psi(x) = e^{-2\pi i l\cdot x} \psi(x - k)$. In this case $[\varphi, \psi](x,\xi) = Z\varphi(x,\xi)\overline{Z\psi(x,\xi)}$, where $x,\xi \in \T^d$, $\varphi, \psi \in L^2(\R^d)$ and
$$
Z\psi(x,\xi) = \sum_{k \in \Z^d}  \psi(k + x)\, e^{-2\pi i k \cdot \xi}
$$
is the Zak transform of the function $\psi \in L^2(\R^d)$ (see for instance \cite{HeilPowell06}).
\end{itemize}

\noindent Given a countable index set $\mathcal{I}$, recall that the system $\big\{ \psi_j \big\}_{j \in \mathcal{I}}$ is a Riesz basis for $\ol{\vsp\big\{ \psi_j \big\}_{j \in \mathcal{I}}}^\Hil$ with constants $0 < A \le B < \infty$ when
$$
A \|c\|^2_{\ell_2(\ind)} \leq \big\| \sum_{j \in \ind} c_j \psi_j \big\|^2_{\Hil} \leq B \|c\|^2_{\ell_2(\ind)}
$$
for all $c = \{c_j\}_{j \in \ind} \in \ell_2(\ind)$. It is a frame for $\ol{\vsp\big\{ \psi_j \big\}_{j \in \mathcal{I}}}^\Hil$ with constants $0 < A \le B < \infty$ when
$$
A \|\varphi\|^2_{\Hil} \leq \sum_{j \in \ind} |\langle \varphi, \psi_j \rangle_{\Hil} |^2 \leq B \|\varphi\|^2_{\Hil}
$$
for all $\varphi \in \ol{\vsp\big\{ \psi_j \big\}_{j \in \mathcal{I}}}^\Hil$. The main results of \cite{HSWW10} can then be summarized as follows. Let $\Gamma$ be a countable discrete abelian group and let $T: \Gamma \to U(\mathcal{H})$ be a dual integrable representation of $\Gamma$. Then $\{T(\gamma) \psi\}_{\gamma \in \Gamma}$ is
\begin{description}
\item[i)] An orthonormal basis for $\psis{\psi}$ if and only if
$$
[\psi,\psi](\alpha) = 1 \quad \mbox{for} \quad \mbox{a.e.} \ \alpha \in \wh{\G}.
$$
\item[ii)] A Riesz basis for $\psis{\psi}$ with frame bounds $0 < A \leq B < \infty$ if and only if
$$
A \leq [\psi,\psi](\alpha) \leq B \quad \mbox{for} \quad \mbox{a.e.} \ \alpha \in \wh{\G}.
$$
\item[iii)] A frame for $\psis{\psi}$ with frame bounds $0 < A \leq B < \infty$ if and only if
$$
A \leq [\psi,\psi](\alpha) \leq B \quad \mbox{for} \quad \mbox{a.e.} \ \alpha \in \textnormal{supp}[\psi,\psi].
$$
\end{description}

The characterization of shift invariant spaces in LCA groups was also addressed in \cite{CabrelliPaternostro10}, in terms of the range function instead of the bracket map. Such tool is strictly related, as shown in \cite{Bownik00}, to the Gramian analysis of \cite{RonShen95, RonShen05}.

A characterization in terms of range functions for translations given by the left regular representation
of nilpotent Lie groups whose irreducible representations are square integrable modulo the center, also called $SI/Z$ Lie groups, has been recently announced in \cite{CMO}, and a partial generalization to such noncommutative setting of the characterization in terms of the bracket map was recently given in \cite{BHM13}. Both results make use of the notion of Fourier transform in terms of the unitary dual of the nonabelian group under consideration. 
\newpage
In order to exploit all the information carried by the group action -- not just its center -- we recall that the compact dual of a nonabelian discrete group can only be understood as a quantum group whose underlying space is a \emph{group von Neumann algebra}, which replaces the former role of Pontryagin duality. This general setting is widely accepted and very well understood in noncommutative geometry and operator algebra \cite{Connes94, BrownOzawa08}. Let $\Gamma$ be a discrete group and $\lr : \Gamma \to \B(\HD)$ its left regular representation given by $\lr(\gamma) \delta_{\gamma'} = \delta_{\gamma\gamma'}$, where the $\delta_\gamma$'s form the unit vector basis of $\HD$. Write $\vn(\Gamma)$ for its group von Neumann algebra, the weak operator closure of the linear span of $\{\lr(\gamma)\}_{\gamma \in \Gamma}$. Given $F \in \mathrm{vNa}(\Gamma)$, we consider the standard normalized trace $\tau(F) = \langle F \delta_\id, \delta_\id \rangle$ where $\id$ denotes the identity element of $\Gamma$. Any such element $F$ has a \emph{Fourier series} 
$$
\sum_{\gamma \in \Gamma} \wh{F}(\gamma) \lr(\gamma) \quad \mbox{with} \quad \wh{F}(\gamma) = \tau(F \lr(\gamma^{-1})) \quad \mbox{so that} \quad \tau(F) = \wh{F}(e).
$$
Let $L^p(\mathrm{vNa}(\Gamma))$ denote the $L^p$ space over the noncommutative measure space $(\mathrm{vNa}(\Gamma), \tau)$ -- so called noncommutative $L^p$ spaces -- equipped with the following norm
$$
\|F\|_p = \Big\| \sum_{\gamma \in \Gamma} \wh{F}(g) \lr(g) \Big\|_p = \Big( \tau \Big[ \hskip1pt \Big| \sum_{\gamma \in \Gamma} \wh{F}(\gamma) \lr(\gamma) \Big|^p \hskip1pt \Big] \Big)^\frac1p = \tau(|F|^p)^\frac1p
$$
where the absolute value $|F| = (F^*F)^{\frac12}$ and the power $p$ are obtained from functional calculus for this (unbounded) operator on the Hilbert space $\HD$. Also set $L^\infty(\mathrm{vNa}(\Gamma)) = \mathrm{vNa}(\Gamma) \subset \B(\HD)$.
We recall that, when $\Gamma$ is abelian,
$$
L^p(\mathrm{vNa}(\Gamma)) = L^p(\widehat{\Gamma})
$$
after identifying $\lr(\gamma)$ with the character $\alpha(\gamma): \widehat{\Gamma} \to \mathbb{T}$ associated to $\gamma \in \Gamma$. There are more explicit descriptions of $L^p(\vn(\Gamma))$ as a space of (so called) measurable operators with finite $L^p$-norm. Noncommutative $L^p$ spaces are Banach spaces for $1 \le p \le \infty$ which share many properties -- also present some pathologies -- with their commutative/classical relatives. This includes duality, H\"older inequalities, real and complex interpolation and Clarkson inequalities among others. We refer to \cite{Nelson74} for an introduction and to \cite{PisierXu03} for an excellent survey paper.
Although these are not closely related subjects, von Neumann algebras have appeared before in the study of orthonormal systems generated by unitary subsets of operators in a Hilbert space. The best references we know are \cite{DaiLarson98} and the survey article \cite{Larson97}. Using this language from noncommutative harmonic analysis, we may generalize the definition of the bracket map defined above to include nonabelian discrete groups.

\begin{DEF}
A unitary representation $T: \Gamma \to U(\mathcal{H})$ of a countable discrete group $\Gamma$ will be called \emph{dual integrable} whenever there exists a map $[\cdot, \cdot] : \Hil \times \Hil \to L^1(\vn(\Gamma))$ such that
\begin{equation}\label{eq:bracketproperty}
\big\langle \varphi , T(\gamma) \psi \big\rangle_{\Hil} = \tau \big( [\varphi,\psi] \lr(\gamma)^* \big)
\end{equation}
holds for every $\varphi, \psi \in \Hil$ and all $\gamma \in \Gamma$.
\end{DEF}

We can now state the main result of this paper.

\begin{TheoA}
Let $\Gamma$ be a countable discrete group acting on some Hilbert space $\mathcal{H}$ via a dual integrable representation $T: \Gamma \to U(\mathcal{H})$. Let $\Id$ stand for the identity operator on $\HD$ and let $\Proj_F$ denote the orthogonal projection in $\HD$ onto $(\Ker F)^\perp$ for any densely defined operator $F$ on $\HD$. Then, given any $\psi \in \Hil$, the system $\{T(\gamma)\psi\}_{\gamma \in \Gamma}$ is
\begin{description}
\item[i)] An orthonormal basis for $\psis{\psi}$ if and only if $[\psi,\psi] = \Id$.
\item[ii)] A Riesz basis for $\psis{\psi}$ with frame bounds $0 < A \leq B < \infty$ if and only if
$$
A \Id \leq [\psi,\psi] \leq B \Id.
$$
\item[iii)] A frame for $\psis{\psi}$ with frame bounds $0 < A \leq B < \infty$ if and only if
$$
A \Proj_{[\psi,\psi]} \leq [\psi,\psi] \leq B \Proj_{[\psi,\psi]} .
$$
\end{description}
\end{TheoA}

After Theorem A, it is essential to construct dual integrable representations to illustrate it. Let us explain how to do it in the concrete framework of unitary representations given by measurable actions $\sigma: \Gamma \curvearrowright (\meas,\mu)$. The action $\sigma$ is called \emph{quasi-$\Gamma$-invariant} when $\mu_\gamma(E) := \mu(\sigma_\gamma(E))$ is absolutely continuous with respect to $\mu$ and admits a positive Radon-Nikodym derivate $J_\sigma(\gamma,\cdot)$ for all elements $\gamma \in \Gamma$. In that case, we may define 
\begin{itemize}
\item The \emph{associated unitary representation}
$$
T_{\sigma}(\gamma) \psi (x) : = J_\sigma(\gamma^{-1}, x)^{\frac12} \psi(\sigma_{\gamma^{-1}} x).\vspace{-6pt}
$$
\item The \emph{noncommutative Zak transform}, defined as the measurable field of operators over $\meas$
given by the formula $\displaystyle Z_\sigma[\psi](x) := \sum_{\gamma \in \Gamma} \Big( T_{\sigma}(\gamma) \psi (x) \Big) \lr(\gamma)^*$.\vspace{-6pt}
\item The \emph{tiling property}. We say that the action $\sigma$ is tiling when there exists a $\mu$-measurable set $C \subset \meas$ such that the family $\{\sigma_\gamma(C)\}_{\gamma \in \Gamma}$ is a disjoint covering of $\meas$ up to a set of zero $\mu$-measure.
\end{itemize}

\begin{TheoB}
Let $\act$ be a quasi-$\Gamma$-invariant action of the countable discrete group $\Gamma$ on the measure space $(\meas,\mu)$, and let $T_\act$ be the associated unitary representation on $L^2(\meas,\mu)$. If $\act$ has the tiling property with tiling set $C$, then
\begin{description}
\item[i)] \ The Zak transform $Z_\sigma$ defines an isometry
$$
Z_\act : L^2(\meas,\mu) \to L^2((C,\mu),L^2(\vn(\Gamma)))
$$
\hspace{-12pt} and satisfies the quasi-periodicity condition $Z_\act[T_\act(\gamma)\psi] = \lr(\gamma) Z_\act[\psi]\,.$
\item[ii)] The representation $T_\sigma$ is dual integrable with bracket
\begin{equation}\label{eq:ZakBracket}
[\varphi, \psi] = \int_C Z_\act[\varphi](x) Z_\act[\psi](x)^* d\mu(x).
\end{equation}
\end{description}
\end{TheoB}

As an application, we will see how Theorem B yields bracket maps in a wide range of scenarios. Other brackets can be constructed ad hoc. Here is a brief list of examples -- by no means exhaustive -- that come to mind. Some of them will be explained in further detail in the body of the paper:

\begin{itemize}
\item {\bf LCA groups.} If $\Gamma$ is abelian, its group von Neumann algebra coincides with the algebra of essentially bounded functions on its dual group. Taking this into account, we shall show how can we recover all the results and examples from \cite{HSWW10} as a particular case of our approach.
 
\item {\bf Shift invariance in $\HD$.} The left regular representation $\Gamma \to U(\HD)$ is always dual integrable. This will be proved in Section \ref{sec:bracket}. In particular, we may characterize which $\psi  \in \HD$ yield ON basis, Riesz basis and frames by translations in the associated invariant subspace of $\HD$.
 
\item {\bf Shift invariance in $L^2(\G)$.} It is even more interesting to replace $\HD$ by a larger Hilbert space in the previous example. Let $\G$ be a locally compact group admitting $\Gamma$ as a discrete subgroup and consider the unitary representation given by left translations of $\Gamma$. We shall show this is again dual integrable for all pairs $(\G,\Gamma)$. Take for instance $\G$ to be the Heisenberg group $\mathbb{H}^n = (\R^n \times \R^n \times \R, \cdot)$ and $\Gamma$ the discrete subgroup with entries in $\Z^n \times \Z^n \times \frac{1}{2} \Z$. Theorem A characterizes those orbits in $L^2(\G)$ which yield ON systems, Riesz systems and frames.

\item {\bf Semidirect products.} We shall also study the case of semidirect products of locally compact groups $S := A \ract \G$, which includes the crystallographic groups, considering unitary representations of a subgroup of $S$ in $L^2(A)$ related to the left regular representation of $A$ and the representation $T_\act$ of $\G$ induced by the action $\act$, both representations in $L^2(A).$ The dual integrability of these representations -- we will consider the quasi-regular and affine representations -- follows from Theorem B. These representations include the ones used in the theory of composite wavelets \cite{WeissWilson01, LaugersenWeaverWeissWilson02, LabateWeiss10}.

\item \textbf{Other examples.} We may consider many other actions of discrete groups in Hilbert spaces. For instance, any map $\Lambda: \{1,2, \ldots, n\} \to U(\mathcal{H})$ yields a unitary action of the free group $\mathbb{F}_n$ with $n$ generators in $\mathcal{H}$. It would also be interesting to consider infinite Coxeter groups trying to exploit their symmetry properties. Another direction to follow is to search for optimal frame constants in the case of finite groups, like the permutation group. A priori, these questions may be faced by showing dual integrability, finding the associated bracket and studying its properties.   
\end{itemize}

\vskip-3pt

The organization of this paper is as follows. In Section \ref{sec:duality} we will review some background on noncommutative integration and group von Neumann algebras that will be used along the paper. Although standard in noncommutative harmonic analysis, we feel this might be useful for the potential audience of this paper. In Section \ref{sec:bracket} we study the main properties of the bracket map and prove  Theorem A. The proof of Theorem B is given in Section \ref{sec:NCZaktransform} together with several equivalent characterizations of the dual integrability -- one of them establishes that $T$ is dual integrable if and only if it is square integrable (but not necessarily irreducible) -- and our examples are explained in Section \ref{Sec:Examples}.

\vskip5pt

\noindent \textbf{Acknowledgement.} D. Barbieri was supported by a Marie Curie Intra European Fellowship (Prop. N. 626055) within the 7th European Community Framework Programme. E. Hern\'andez and J. Parcet were supported by Grant MTM2010-16518 (Ministerio de Econom\'ia y Competitividad, Spain) and J. Parcet was also supported by ERC Starting Grant 256997-CZOSQP and by ICMAT Severo Ochoa Grant SEV-2011-0087 (Spain).

\section{Noncommutative Fourier series}\label{sec:duality}

The problem of Fourier duality for discrete (and more general locally compact) groups has been addressed in terms of their von Neumann algebra in several contexts \cite{Kunze58, Eymard64, EnockSchwartz92, Connes94}.
This approach directly extends the ordinary Pontryagin duality for LCA groups: if $\G$ is a LCA group, the Fourier transform of an integrable function $f$ on $\G$ is defined on the character group $\wh{\G}$ as
$$
\F f (\alpha) := \int_\G f(g) \ol{\alpha(g)} dg\, ,
$$
where $dg$ stands for the Haar measure of $\G$. The map $\F$ extends to an isometry of $L^2(\G)$ onto $L^2(\wh\G)$ which makes the operator $L_f$ of convolution by $f$ in $L^2(\G)$ correspond to the operator $M_{\F f}$ of multiplication by $\F f$ in $L^2(\wh\G)$ as
$$
M_{\F f} := \F L_f \F^{-1} .
$$
Since for any measurable $f$ such that $L_f$ is a closed densely defined operator on $L^2(\G)$ one can define the Fourier transform of $f$ as such multiplier, for general locally compact groups a natural choice is then to consider directly the operator $L_f$ as the Fourier transform of $f$. A characterization of the von Neumann algebra $\vn(\Gamma)$ of a discrete group $\Gamma$ is indeed that of being isomorphic to the left convolution algebra of $\HD$.

\subsection{Noncommutative integration}

In this section we review some basic notions for operators on a Hilbert space, which can be found e.g. in \cite[Chapt. 3]{Conway00} and \cite[Vol. 2, Chapt. 6]{KadisonRingrose83}, and for von Neumann algebras and noncommutative $L^p$ spaces, see e.g. \cite{KadisonRingrose83}, \cite{Nelson74} and \cite{PisierXu03}.

Let $F$ be a densely defined closed operator on a Hilbert space $\Hil$, and denote with $\mathcal{D}(F)$ the domain of $F$. $F$ is said to be \emph{positive} if
$$
\langle u, F u\rangle_{\Hil} \geq 0 \quad \forall \ u \in \mathcal{D}(F) .
$$
This is equivalent to say that for any positive operator $F$ there exists a densely defined operator $x$ on $\Hil$ such that $F = x^*x$. Any time we write an operator inequality such as $F_1 \geq F_2$ we mean that $F_1 - F_2 \geq 0$.

The right and left \emph{support} of $F$, indicated respectively with $r_F$ and $l_F$, are defined as the minimal orthogonal projections such that
$$
F r_F = F \ \ \ \textrm{and} \ \ \ l_F F = F .
$$
Explicitly, $r_F = \Proj_{(\Ker(F))^\bot}$ and $l_F = \Proj_{\ol{\Ran(F)}}$ are the orthogonal projections onto the orthogonal complement of the kernel and on the closure of the range. If $F = F^*$, then $r_F = l_F$ and in such a case we indicate the support with $s_F$. Moreover, since $\Ker(F^*F) = \Ker(F)$, we have that $r_F = s_{|F|}$ where $|F| = (F^*F)^\frac12$.

A \emph{partial isometry} $u$ on a Hilbert space is an operator whose restriction to $(\Ker(u))^\bot$ is an isometry with the closure of the range. It is characterized by the properties $u = u u^* u$ or $u^* u = \Proj_{(\Ker(u))^\bot}$ .

For any closed and densely defined $F$ there exists a unique partial isometry $u_F$ satisfying $\Ker(u_F) = \Ker(F)$ and such that
$$
F = u_F |F|
$$
which is called the \emph{polar decomposition} of $F$. In particular $u_F^* u_F^{\phantom{*}} = r_F$.
\newpage
Part of von Neumann algebra theory has evolved as the noncommutative form of measure theory and integration. 
A \emph{von Neumann algebra} $\M$ (see e.g. \cite[Vol.1, Chapt. 5]{KadisonRingrose83}) is a
weak-operator closed complex unital algebra of bounded linear operators on a complex Hilbert space that is closed under taking the adjoint.  The positive cone $\M_+$ is the set of positive operators in $\M$ and a trace $\tau: \M_+ \to [0,\infty]$ is a linear map with the \emph{tracial property}
$$
\tau(F^*F) = \tau(FF^*).
$$
The trace $\tau$ plays the r\^ole of the integral in the classical case. It is normal if $\sup_\alpha \tau(F_\alpha) = \tau(\sup_\alpha F_\alpha)$ for bounded increasing nets $\{F_\alpha\}$ in $\M_+$; it is semifinite if for any non-zero $F \in \M_+$ there exists $0 < F' \le F$ such that $\tau(F') < \infty$; and it is faithful if $\tau(F) = 0$ implies that $F = 0$. A von Neumann algebra is semifinite when it admits a normal semifinite faithful (n.s.f. in short) trace $\tau$. Any operator $F$ is a linear combination $F_1 - F_2 + iF_3 - iF_4$ of four positive operators. Thus, we can extend $\tau$ to the whole algebra $\M$ and the tracial property can be restated in the familiar form $\tau(F_1F_2) = \tau(F_2F_1)$.
We will always consider semifinite von Neumann algebras equipped with a n.s.f. trace, and refer to the pair $(\M,\tau)$ as a \emph{noncommutative measure space}. Note that commutative von Neumann algebras correspond to classical $L^\infty$ spaces \cite{PisierXu03}.

Let $\mathcal{S}_\M^+$ be the set of all $F \in \M_+$ such that $\tau(s_F) < \infty$ and set $\mathcal{S}_\M$ to be the linear span of $\mathcal{S}_\M^+$. If we write $|F|=\sqrt{F^*F}$, we can use the spectral measure $d\gamma: \R_+ \to \B(\Hil)$ of $|F|$ to define
$$
|F|^p = \int_{\R_+} s^p \, d \gamma(s) \quad \mbox{for} \quad 0 < p < \infty.
$$
We have $F \in \mathcal{S}_\M \Rightarrow |F|^p \in \mathcal{S}_\M^+ \Rightarrow \tau(|F|^p) < \infty$. If we set $\|F\|_p = \tau( |F|^p )^{\frac1p}$, we obtain a norm in $\mathcal{S}_\M$ for $1 \le p < \infty$ and a $p$-norm for $0 < p < 1$. Using that $\mathcal{S}_\M$ is an involutive strongly dense ideal of $\M$, we define the \emph{noncommutative $L^p$ space} $L^p(\M)$ associated to the pair $(\M, \tau)$ as the completion of $(\mathcal{S}_\M, \| \ \|_p)$. On the other hand, we set $L^\infty(\M) = \M$ equipped with the operator norm. Elements of $L^p(\M)$ can also be described as measurable operators $F$ affiliated to $(\M,\tau)$ for which $\tau(|F|^p)$ is finite (see e.g. \cite{PisierXu03} for the definitions). The space $L^2(\M)$ is a Hilbert space whose scalar product can be obtained by polarization of the corresponding norm: for $F_1, F_2 \in L^2(\M)$
$$
\langle F_1, F_2\rangle_2 = \tau(F_2^* F_1).
$$
Many fundamental properties of classical $L^p$ spaces such as duality, real and complex interpolation, H\"older inequalities hold in this setting, and we refer to \cite{Nelson74, PisierXu03} for more information and historical references.

Note that classical $L^p(\Omega,\mu)$ spaces are denoted in the noncommutative terminology as $L^p(\M)$ where $\M$ is the commutative von Neumann algebra $L^\infty(\Omega,\mu)$.

\subsection{Group von Neumann algebras}

Let $\Gamma$ be a discrete group, and let $\B(\HD) \supset U(\HD)$ denote respectively the bounded and unitary operators on $\HD$. Let us denote with $\id$ the identity of $\Gamma$, and with $\{\delta_\gamma\}_{\gamma \in \Gamma}$ the canonical basis on $\HD$. We will also denote with $\lr$ the left regular representation of $\Gamma$, given by $\lr(\gamma) \delta_{\gamma'} = \delta_{\gamma\gamma'}$.

In analogy with classical Fourier analysis on $\T$, we will call \emph{trigonometric polynomials} the operators obtained by finite linear combinations of the left regular representation. The \emph{von Neumann algebra associated to $\Gamma$} is then defined as the closure of trigonometric polynomials in the weak operator topology
$$
\vn(\Gamma) := \ol{\vsp\{\lr(\gamma)\}_{\gamma \in \Gamma}}^{\mathrm{WOT}} .
$$
This amounts to say that $F \in \vn(\Gamma)$ if and only if there exists a sequence $\{F_n\}$ of trigonometric polynomials that converges to $F$ in the weak operator topology of $\HD$.
Given $F \in \mathrm{vNa}(\Gamma)$, we consider the standard normalized trace
$$
\tau(F) = \langle F \delta_e, \delta_e \rangle
$$
where $e$ denotes the identity element of $\Gamma$. It is indeed tracial and it is easy to see that it is normal, faithful and finite, i.e. $\tau(F^*F) < \infty$ for all $F \in \vn(\Gamma)$.

Since $(\vn(\Gamma),\tau)$ is a noncommutative measure space, the associated noncommutative $L^p$ spaces over $\vn(\Gamma)$ are defined as in the previous section. In particular, we note that $L^p(\vn(\Gamma)) \equiv \ol{\vn(\Gamma)}^{\|\cdot\|_p}$ when $1 \leq p < \infty$. Moreover, since $\tau$ is finite, H\"older inequality implies $L^q(\vn(\Gamma)) \subset L^p(\vn(\Gamma))$ for all $1 \leq p \leq q \leq \infty$.
Note that, while $L^\infty(\vn(\Gamma)) \equiv \vn(\Gamma) \subset \B(\HD)$, operators in $L^p(\vn(\Gamma))$ for $p < \infty$ in general need not be bounded.

Moreover we observe that $\vn(\Gamma)$ is the Banach dual of $L^1(\vn(\Gamma))$, so the convergence in the weak operator topology is equivalent to the weak-$*$ convergence, that is $\tau((F - F_n)G) \to 0$ as $n \to \infty$ for all $G \in L^1(\vn(\Gamma))$. Note also that since trigonometric polynomials define a unital $C^*$-subalgebra of $\B(\HD)$, by the Double Commutant Theorem this closure coincides with the closure in the strong operator topology (see e.g. \cite{Conway00}).

For $F \in \vn(\Gamma)$, its \emph{Fourier coefficients} are
$$
\wh{F}(\gamma) := \tau(\lr(\gamma)^*F)
$$
and one can write
$$
F = \sum_{\gamma \in \Gamma} \wh{F}(\gamma) \lr(\gamma) \, , 
$$
where convergence is understood in the weak operator topology.
The operator $F$ is then a bounded left convolution operator on $\HD$ with convolution kernel given by $\wh{F} = F\delta_\id$:
$$
F \psi (\gamma') = \sum_{\gamma \in \Gamma} \wh{F}(\gamma) \psi(\gamma^{-1}\gamma') = (\wh{F} \ast \psi)(\gamma') \quad \forall \ \psi \in \HD \, , \ \forall \ \gamma' \in \Gamma.
$$
Note also that, since $\tau$ is finite, H\"older inequality implies that the Fourier coefficients are well-defined for all $F \in L^p(\vn(\Gamma))$ for $1 \leq p \leq \infty$ and in particular if $F \in L^1(\vn(\Gamma))$ then $\wh{F} \in \ell_\infty(\Gamma)$. Moreover, as $\tau(\lambda(\gamma'\gamma^{-1})^*F) = \langle F\delta_\gamma, \delta_{\gamma'}\rangle_{\HD}$ is well defined for all $\gamma, \gamma' \in \Gamma$ when $F \in L^p(\vn(\Gamma))$, then $F$ is defined on the orthogonal basis $\{\delta_\gamma\}_{\gamma \in \Gamma}$ and hence on the dense domain of finite sequences.

\begin{lem}[Uniqueness]\label{lem:uniqueness}
Let $F \in L^1(\vn(\Gamma))$. If $\wh{F}(\gamma) = 0$ for all $\gamma \in \Gamma$, then $F = 0$.
\end{lem}
\begin{proof}
Let $\mathbf{j} : L^1(\vn(\Gamma)) \to \vn(\Gamma)^*$ be the natural injection of $L^1(\vn(\Gamma))$ in its double dual, that maps $F \in L^1(\vn(\Gamma))$ to the continuous functional
$$
\mathbf{j}(F)\Phi := \tau(F\Phi) \quad \forall \ \Phi \in \vn(\Gamma).
$$

Let us now suppose that $\tau(F\lr(\gamma)^*) = \wh{F}(\gamma) = 0$ for all $\gamma \in \Gamma$, and let $\Phi \in \vn(\Gamma)$ be a trigonometric polynomial, i.e. for some finite $\Omega \subset \Gamma$
$$
\Phi = \sum_{\gamma \in \Omega} \phi(\gamma) \lr(\gamma) .
$$
Then
$$
\mathbf{j}(F)\Phi = \tau(F\Phi) = \sum_{\gamma \in \Omega} \phi(\gamma) \tau(F\lr(\gamma)) = 0 .
$$
Since any element of $\vn(\Gamma)$ is the weak-$*$ limit of such polynomials, by weak-$*$ continuity $\mathbf{j}(F) = 0$ and by injectivity $F = 0$.
\end{proof}


Observe that, since
$$
\langle \lr(\gamma), \lr(\gamma')\rangle_2 = \langle \delta_{\gamma'} , \delta_{\gamma} \rangle_{\HD} = \delta_{\gamma,\gamma'} \ ,
$$
then $\{\lambda(\gamma)\}_{\gamma \in \Gamma}$ is an orthonormal system. Since $L^2(\vn(\Gamma)) \subset L^1(\vn(\Gamma))$, by the uniqueness Lemma \ref{lem:uniqueness} one has that $\langle F, \lambda(\gamma)\rangle_2 = 0$ for all $\gamma \in \Gamma$ implies $F = 0$, which means that $(\vsp\{\lambda(\gamma)\}_{\gamma \in \Gamma})^\bot = \{0\}$ in $L^2(\vn(\Gamma))$. Hence trigonometric polynomials are dense in $L^2(\vn(\Gamma))$, and $\{\lambda(\gamma)\}_{\gamma \in \Gamma}$ is an orthonormal basis. This allows to prove the following Plancherel-type result.
\begin{lem}[Plancherel]\label{lem:Plancherel}\ \\
\vspace{-12pt}
\begin{itemize}
\item[i.] Let $F$ be in $L^2(\vn(\Gamma))$. Then $\{\wh{F}(\gamma)\} \in \HD$ and $\|F\|_2 = \|\wh{F}\|_{\HD}$.
\item[ii.] Let $\{\phi(\gamma)\} \in \HD$. Then ${\sum_{\gamma \in \Gamma}} \phi(\gamma) \lr(\gamma)$ converges in the $\|\cdot\|_2$ norm to an operator $\Phi \in L^2(\vn(\Gamma))$ such that $\wh{\Phi}(\gamma) = \phi(\gamma)$ for all $\gamma \in \Gamma$.
\end{itemize}
\end{lem}

Note that this also implies that trigonometric polynomials are dense in $L^p(\vn(\Gamma))$ for all $1 \leq p < \infty$. Indeed, for $p < 2$ by H\"older inequality and the finiteness of $\tau$ one has $\|\Phi\|_p \leq \|\Phi\|_2$ while for $p > 2$ one has
$$
\|\Phi\|_p^p = \tau(|\Phi|^{p-2}|\Phi|^2) \leq \|\Phi\|_\infty^{p - 2} \|\Phi\|_2^2
$$
for all $\Phi \in \vn(\Gamma)$. Let $\{F_n\}$ be a sequence of trigonometric polynomials converging weak-$*$ (and hence in the $\|\cdot\|_2$ norm) to $F \in \vn(\Gamma)$. Then by choosing $\Phi = F - F_n$ in the preceding inequalities one has that $F_n$ converges to $F$ in $L^p(\vn(\Gamma))$, so that $\ol{\vn(\Gamma)}^{\|\cdot\|_p} \equiv \ol{\vsp\{\lambda(\gamma)\}_{\gamma \in \Gamma}}^{\|\cdot\|_p}$, for all $1 \leq p < \infty$.

Moreover, since for $\{\wh{F}(\gamma)\} \in \ell_1(\Gamma)$
\begin{equation}\label{eq:Pontryagin}
\|F\|_\infty = \Big\|\sum_{\gamma \in \Gamma} \wh{F}(\gamma) \lr(\gamma)\Big\|_{\HD \to \HD} \leq \sum_{\gamma \in \Gamma} |\wh{F}(\gamma)| = \|\wh{F}\|_{\ell_1(\Gamma)} \ ,
\end{equation}
by interpolation with the case $p = 2$ one gets the Hausdorff-Young inequalities
$$
\|F\|_q \leq \|\wh{F}\|_{\ell_p(\Gamma)} \quad \textnormal{for} \quad \frac1p + \frac1q = 1 \, , \ 1 \leq p \leq 2 .
$$

We will also make use of the following lemma.
\begin{lem}\label{lem:Kaplansky}
Let $G$ be an $L^1(\vn(\Gamma))$ selfadjoint operator such that
$$
\tau(|F|^2G) \geq 0
$$
for all trigonometric polynomials $F$. Then $G \geq 0$.
\end{lem}
\begin{proof}
In order to prove this lemma it suffices to prove that
\begin{equation}\label{eq:KaplanskyClaim}
\tau(qG) \geq 0
\end{equation}
for all orthogonal projections $q \in \vn(\Gamma)$. Indeed if that holds, then we can choose $q_0 := \chi_{(-\infty,0)}(G)$, the spectral projection associated with $G$ of the Borel set $(-\infty,0)$, and obtain a contradiction unless $q_0 = 0$. That the orthogonal projection $q_0$ actually belongs to $\vn(\Gamma)$ is a consequence of \cite[Proposition 5.3.4]{Sunder97}, since any $G \in L^1(\vn(\Gamma))$ is affiliated to $\vn(\Gamma)$ and hence \cite[ii) Proposition 5.3.4]{Sunder97} holds.

Claim (\ref{eq:KaplanskyClaim}) is proved if we show that for all orthogonal projections $q \in \vn(\Gamma)$ there exists a sequence $\{F_n\}$ of trigonometric polynomials such that $\{|F_n|^2\}$ converges in the weak-$*$ topology to $q$, since then using the normality of $\tau$ we would get
$$
\tau(qG) = \tau(|F_n|^2G) + \tau((q - |F_n|^2)G) = \lim_{n \to \infty} \tau(|F_n|^2G) \geq 0 .
$$
In order to do so, let us first observe that since $q$ is an orthogonal projection
\begin{eqnarray*}
q - |F_n|^2 & = & q^* q - F_n^* F_n = q^* q - q^* F_n + q^* F_n - F_n^* F_n\\
& = & q^* (q - F_n) - (q - F_n)^* (q - F_n) + (q - F_n)^* q\\
& = & q^* A_n + A_n^* q - A_n^* A_n
\end{eqnarray*}
where we have called $A_n = q - F_n$. This implies that for $|F_n|^2$ to converge in the weak-$*$ topology to $q$ it suffices that $F_n$ converges strongly to $q$, because in that case both $A_n$ and $A_n^*$ would converge in the weak-$*$ topology to 0 and the same would be true for $A_n^* A_n$, since
$$
\langle A_n^* A_n u, u\rangle_{\HD} = \|(q - F_n)u\|^2_{\HD} .
$$
The existence of a sequence of trigonometric polynomials converging strongly to an orthogonal projection is then a consequence of Kaplansky density theorem (see e.g. \cite[Theorem 44.1]{Conway00}).
\end{proof}

\section{Bracket map}\label{sec:bracket}

Recall the definition of dual integrability given in the introduction.
We start with a relevant example of dual integrable representation, provided by the left regular representation itself.

\begin{lem}\label{lem:dualint}
The left regular representation $\lr$ of a discrete countable group $\Gamma$ on $\HD$ is dual integrable and the bracket map is given by
$$
[\psi_1,\psi_2] = \Psi_1 \Psi_2^*
$$
where $\Psi_i$ is characterized by $\wh{\Psi_i}(\gamma) = \psi_i(\gamma)$ for all $\gamma \in \Gamma$, $i = 1,2$.
\end{lem}
\begin{proof}
Given $\psi \in \HD$, by Plancherel Lemma \ref{lem:Plancherel} the operator
$$\Psi = \sum_{\gamma \in \Gamma} \psi(\gamma) \lr(\gamma)$$
is an element of $L^2(\vn(\Gamma))$ with $\|\Psi\|_2 = \|\psi\|_{\HD}$.
By H\"older inequality, given $\psi_1, \psi_2 \in \HD$ the operator $\Psi_1 \Psi_2^*$ is then in $L^1(\vn(\Gamma))$. Moreover
\begin{eqnarray*}
\tau(\lr(\gamma)^*\Psi_1\Psi_2^*) & = & \langle \Psi_1 \Psi_2^* \delta_\id, \delta_\gamma\rangle_{\HD}\\
& = & \sum_{\gamma_1 \in \Gamma} \sum_{\gamma_2 \in \Gamma} \psi_1(\gamma_1) \ol{\psi_2(\gamma_2)} \langle \lr(\gamma_1) \lr(\gamma_2^{-1}) \delta_\id, \delta_\gamma\rangle_{\HD}\\
& = & \sum_{\gamma_1 \in \Gamma} \sum_{\gamma_2 \in \Gamma} \psi_1(\gamma_1) \ol{\psi_2(\gamma_2)} \langle \delta_{\gamma_1\gamma_2^{-1}}, \delta_\gamma\rangle_{\HD}\\
& = & \sum_{\gamma_1 \in \Gamma} \psi_1(\gamma_1) \ol{\psi_2(\gamma^{-1}\gamma_1)} = \langle \psi_1,  \lr(\gamma) \psi_2\rangle_{\HD}
\end{eqnarray*}
so the proof follows by (\ref{eq:bracketproperty}) and the uniqueness of Fourier coefficients.
\end{proof}

\subsection{Properties of the bracket map}\label{sec:structure}
As a consequence of the definition of dual integrability, we have that the bracket map has the following properties.

\begin{prop}\label{prop:properties}
Let $T$ be a dual integrable representation on $\Hil$, and let $\psi, \psi_1, \psi_2 \in \Hil$. Then the following properties hold
\begin{itemize}
\item[\textnormal{I)}] the bracket is sesquilinear and satisfies $[\psi_1, \psi_2]^* = [\psi_2, \psi_1]$
\item[\textnormal{II)}] $[T(\gamma)\psi_1, \psi_2] = \lr(\gamma)[\psi_1, \psi_2]$ \ , \ \ $[\psi_1, T(\gamma)\psi_2] = [\psi_1, \psi_2]\lr(\gamma)^*$ \ , \ \  $\forall \ \gamma \in \Gamma$
\item[\textnormal{III)}] $[\psi,\psi]$ is nonnegative, and $\|[\psi, \psi]\|_1 = \|\psi\|^2_{\Hil}$ .
\end{itemize}
\end{prop}
\begin{proof}
To prove Property I) we first note that sesquilinearity follows from the definition of the bracket (see (\ref{eq:bracketproperty})), the linearity of $\tau$ and the uniqueness Lemma \ref{lem:uniqueness}. Moreover, the trace $\tau$ satisfies
$$
\tau(F^*) = \langle F^* \delta_e , \delta_e \rangle_{\HD} = \overline{\langle  \delta_e , F^*\delta_e \rangle_{\HD}} = \overline{\langle F\delta_e, \delta_e \rangle_{\HD}} = \overline{\tau(F)}
$$
so that, by traciality and using that $\lr$ is a unitary homomorphism
$$
\tau([\psi_1, \psi_2]^* \lr(\gamma)^*) = \ol{\tau(\lr(\gamma) [\psi_1, \psi_2])} = \ol{\tau([\psi_1, \psi_2] \lr(\gamma))}
 = \ol{\tau([\psi_1, \psi_2] \lr(\gamma^{-1})^*)} .
$$
By the defining relation (\ref{eq:bracketproperty}) of the bracket and using that $T$ is a unitary homomorphism we then obtain
$$
\tau([\psi_1, \psi_2]^* \lr(\gamma)^*) = \ol{\langle \psi_1 , T(\gamma^{-1}) \psi_2 \rangle_{\Hil}} = \langle \psi_2 , T(\gamma) \psi_1 \rangle_{\Hil} = \tau([\psi_2, \psi_1] \lr(\gamma)^*)
$$
so that Property I) holds due to the uniqueness Lemma \ref{lem:uniqueness}.
\newpage
Property II) can be proved using the same arguments, since
\begin{eqnarray*}
\tau([\psi_1, T(\gamma_0)\psi_2] \lr(\gamma)^*) & = & \langle \psi_1 , T(\gamma) T(\gamma_0) \psi_2 \rangle_{\Hil} = \langle \psi_1 , T(\gamma \gamma_0) \psi_2 \rangle_{\Hil}\\
& = & \tau([\psi_1, \psi_2] \lr(\gamma \gamma_0)^*) = \tau([\psi_1, \psi_2] \lr(\gamma_0)^*\lr(\gamma)^*) .
\end{eqnarray*}
The other equality is proved from this result and Property I).

To prove Property III) we first note that, by (\ref{eq:bracketproperty})
$$
\|\psi\|^2_{\Hil} = \langle \psi , \psi \rangle_{\Hil} = \tau([\psi,\psi])
$$
so, by the definition of $L^1(\vn(\Gamma))$, the proof is concluded if we can show that
$[\psi,\psi]$ is positive. Since we already know that $[\psi,\psi]$ is selfadjoint and belongs to $L^1(\vn(\Gamma))$, by Lemma \ref{lem:Kaplansky} it is enough to prove that
$$
\tau(|F|^2[\psi,\psi]) \geq 0
$$
for all trigonometric polynomial $F$. Let then $\Omega$ be a finite subset of $\Gamma$. By the dual integrability of $T$, using properties I), II)  and the traciality of $\tau$ we have
\begin{align}
\tau\Big(\Big|\sum_{\gamma \in \Omega} & a_\gamma \lr(\gamma)\Big|^2[\psi,\psi]\Big) = \tau\Big(\sum_{\gamma_1 , \gamma_2 \in \Omega} a_{\gamma_1}\ol{a_{\gamma_2}} \lr(\gamma_1) [\psi,\psi] \lr(\gamma_2)^*\Big)\nonumber\\
& = \tau\Big(\Big[\sum_{\gamma_1 \in \Omega} a_{\gamma_1} T(\gamma_1)\psi,\sum_{\gamma_2 \in \Omega} a_{\gamma_2} T(\gamma_2)\psi\Big]\Big)
= \Big\|\sum_{\gamma \in \Omega} a_\gamma T(\gamma)\psi\Big\|_\Hil^2\label{eq:finitesum}
\end{align}
hence proving the positivity of the bracket.
\end{proof}

Let $[\cdot , \cdot]$ be the bracket of a dual integrable representation $T$ on the Hilbert space $\Hil$. Then for $\psi \in \Hil$ the functional
\begin{equation}\label{eq:semin}
\|F\|_{2,\psi} = \Big(\tau(|F|^2[\psi,\psi])\Big)^\frac12 = \|F[\psi,\psi]^\frac12\|_2 \ , \quad F \in \vn(\Gamma)
\end{equation}
defines a seminorm on $\vn(\Gamma)$.

\begin{lem}\label{lem:weighted}
Let $N_\psi$ be the null space associated with the seminorm (\ref{eq:semin}), and let us call $\widetilde{\mathfrak{h}} := \vn(\Gamma) / N_\psi$. Then $\widetilde{\mathfrak{h}}$ can be identified with
$$
\mathfrak{h} = \{B \in \B(\Hil) : B = F s_{[\psi,\psi]} \ \textnormal{for some} \ F \in \vn(\Gamma) \}
$$
and in this sense we use the notation $\mathfrak{h} = \vn(\Gamma) s_{[\psi,\psi]}$.
\end{lem}
\begin{proof}
Let $\Phi : \vn(\Gamma) \to \mathfrak{h}$ be the linear surjective map defined by
$$
\Phi : F \mapsto F s_{[\psi,\psi]} .
$$
The desired result is a consequence of the fact that
$$
\Ker\, \Phi = N_\psi
$$
since in this case the map $\widetilde\Phi : \widetilde{\mathfrak{h}} \to \mathfrak{h}$ given by $\widetilde\Phi([F]) = Fs_{[\psi,\psi]}$ is a well defined linear bijection, where $[F]$ stands for the equivalence class of $F \in \vn(\Gamma)$ in $\widetilde{\mathfrak{h}}$. Let us then take $F \in \Ker\, \Phi$; then $F s_{[\psi,\psi]} = 0$, which implies $0 = F s_{[\psi,\psi]} [\psi,\psi]^{\frac12} = F [\psi,\psi]^{\frac12}$ and hence $F \in N_\psi$. Conversely, if $F \in N_\psi$ then $F [\psi,\psi]^{\frac12} = 0$; therefore $0 = F [\psi,\psi]^{\frac12} [\psi,\psi]^{-\frac12}s_{[\psi,\psi]} = F s_{[\psi,\psi]}$, so $F \in \Ker\, \Phi$.
\end{proof}

We will denote with $L^2(\vn(\Gamma),[\psi,\psi])$ the closure of $\mathfrak{h}$ in the $\|\cdot\|_{2,\psi}$ norm
$$
L^2(\vn(\Gamma),[\psi,\psi]) := \ol{\,\mathfrak{h}\ }^{\|\cdot\|_{2,\psi}}
$$
noting that this is a Hilbert space, with scalar product
$$
\langle F_1, F_2\rangle_{2,\psi} := \tau(F_2^* F_1 [\psi,\psi]) = \langle F_1 [\psi,\psi]^\frac12, F_2 [\psi,\psi]^\frac12\rangle_{2} .
$$

\begin{prop}\label{prop:isometry}
Let $T$ be a dual integrable representation of a discrete countable group $\Gamma$ on the Hilbert space $\Hil$. Then to each $0 \neq \psi \in \Hil$ there corresponds an isometric isomorphism
$$
S_\psi : \psis{\psi} \rightarrow L^2(\vn(\Gamma),[\psi,\psi])
$$
satisfying
$$
S_\psi[T(\gamma)\psi] = \lr(\gamma) .
$$
\end{prop}
\begin{proof}
For any nonzero $\psi \in \Hil$, let us define the map from $\vsp\{T(\gamma)\psi\}_{\gamma \in \Gamma}$ to the space of trigonometric polynomials given by
$$
S_\psi : \sum_{\gamma \in \Omega} a_\gamma T(\gamma)\psi \mapsto \sum_{\gamma \in \Omega} a_\gamma \lr(\gamma)
$$
where $\Omega$ is a finite subset of $\Gamma$.
By (\ref{eq:finitesum}), $S_\psi$ is a densely defined isometry which extends to an isometry on $\psis{\psi}$ by density.

In order to prove surjectivity, let us proceed by contradiction and assume that there exists $F \in L^2(\vn(\Gamma),[\psi,\psi])$ such that
\begin{equation}\label{eq:contradictionisom}
\langle F , S_\psi[\varphi]\rangle_{2,\psi} = 0 \ \textnormal{for all} \ \varphi \in \psis{\psi}. 
\end{equation}
This holds in particular if $\varphi = T(\gamma)\psi$ for all $\gamma \in \Gamma$, but since $S_\psi[T(\gamma)\psi] = \lr(\gamma)$
then (\ref{eq:contradictionisom}) implies
$$
\tau(\lr(\gamma)^*F [\psi,\psi]) = \langle F , \lr(\gamma)\rangle_{2,\psi} = 0 .
$$
Since both $F[\psi,\psi]^\frac12$ and $[\psi,\psi]^\frac12$ belong to $L^2(\vn(\Gamma))$, then $F [\psi,\psi]$ belongs to $L^1(\vn(\Gamma))$ and by the uniqueness Lemma \ref{lem:uniqueness} we get
$$
F [\psi,\psi] = 0 .
$$
This implies in particular that
$$
0 = F [\psi,\psi] [\psi,\psi]^{-1} s_{[\psi,\psi]} = F s_{[\psi,\psi]}
$$
and since $F \in L^2(\vn(\Gamma),[\psi,\psi])$, by construction we have that $F s_{[\psi,\psi]} = F$.
\end{proof}

\subsection{Proof of the main result}\label{sec:proof}

The notion of bracket map allows to characterize orthonormal bases, Riesz bases and frames for the principal shift-invariant space $\psis{\psi}$ of $\Hil$. 
Such characterizations are given in Theorem A and we will give the proof of these results in this section.

\begin{proof}[Proof of Theorem A, i.]
Let us first suppose that $\{T(\gamma)\psi\}_{\gamma \in \Gamma}$ is an orthonormal system. Then by (\ref{eq:bracketproperty})
$$
\tau([\psi,\psi]\lr(\gamma)^*) = \delta_{\gamma,\id} \quad \forall \, \gamma \in \Gamma .
$$
Since $\tau(\lr(\gamma')\lr(\gamma)^*) = \delta_{\gamma,\gamma'}$, by the uniqueness lemma
$
[\psi,\psi] = \lr(\id) = \Id .
$
Conversely, if $[\psi,\psi] = \Id$, by (\ref{eq:bracketproperty})
\begin{displaymath}
\langle \psi , T(\gamma) \psi \rangle_{\Hil} = \tau(\lr(\gamma)^*) = \tau(\lr(\id)\lr(\gamma)^*) = \delta_{\gamma,\id} .
\qedhere
\end{displaymath}
\end{proof}

\begin{proof}[Proof of Theorem A, ii.]
In order to prove this, we show that the following are equivalent.
\begin{itemize}
\item[a)] $\{T(\gamma)\psi\}_{\gamma \in \Gamma}$ is a Riesz basis for $\psis{\psi}$ with constants $A$ and $B$.
\item[b)] For all $F \in \vsp\{\lr(\gamma)\}_{\gamma \in \Gamma}$
$$
A \|F\|_2^2 \leq \tau(|F|^2\,[\psi,\psi]) \leq B \|F\|_2^2 .
$$
\item[c)] The following inequality holds
$$
A \Id \leq [\psi,\psi] \leq B \Id .
$$
\end{itemize}
$a) \Leftrightarrow b)$. Recall that $\{T(\gamma)\psi\}_{\gamma \in \Gamma} \subset \Hil$ is a Riesz system with constants $A, B$ if for any finite sequence $\{a_\gamma\}_{\gamma \in \Omega}$, where $\Omega$ is any finite subset of $\Gamma$, it holds
$$
A \sum_{\gamma \in \Omega} |a_\gamma|^2 \leq \|\sum_{\gamma \in \Omega} a_\gamma T(\gamma)\psi\|^2_{\Hil} \leq B \sum_{\gamma \in \Omega} |a_\gamma|^2 .
$$
Moreover, by definition, the system $\{T(\gamma)\psi\}_{\gamma \in \Gamma}$ is complete in $\psis{\psi}$, therefore it is a Riesz basis for $\psis{\psi}$.

Let $F = \displaystyle\sum_{\gamma \in \Omega} a_\gamma \lr(\gamma)$, with $\Omega \subset \Gamma$ finite, be an element of $\vsp\{\lr(\gamma)\}_{\gamma \in \Gamma}$.
The equivalence of $a)$ and $b)$ can then be obtained since $\|F\|_2^2 = \sum_{\gamma \in \Omega} |a_\gamma|^2$ while, by (\ref{eq:finitesum}), it holds
$$
\Big\| \sum_{\gamma \in \Omega} a_\gamma T(\gamma)\psi \Big\|_\Hil^2 = \tau(|F|^2[\psi,\psi]) .
$$

$c) \Rightarrow b)$. Let us consider the right inequality in $c)$
$$
B\Id - [\psi,\psi] \geq 0
$$
and let us call it $x^*x$. Then for any $F \in \vsp\{\lr(\gamma)\}_{\gamma \in \Gamma}$ we get
$$
0 \leq (xF^*)^*xF^* = F (B \Id - [\psi,\psi]) F^* .
$$
By the traciality of $\tau$ this implies
$$
0 \leq \tau\Big( F (B \Id - [\psi,\psi]) F^* \Big) = \tau\Big( |F|^2 (B \Id - [\psi,\psi]) \Big)
$$
which proves the right inequality in $b)$. The other inequality follows by similar argument.

$b) \Rightarrow c)$. This is a direct consequence of Lemma \ref{lem:Kaplansky}. Indeed assume the left inequality in $b)$, then since $G = [\psi,\psi] - A\Id$ is an $L^1(\vn(\Gamma)))$ selfadjoint operator, then $[\psi,\psi] \geq A\Id$, and similarly for the other inequality.
\end{proof}

For the next step, we note that while the approach followed and the result obtained are in the spirit of \cite{HSWW10}, the scheme of the proof partially resembles that of \cite{BenedettoLi98} for the case of integer translations over $\R^n$.

\begin{proof}[Proof of Theorem A, iii.]
The proof is done by showing that the following are equivalent.
\begin{itemize}
\item[a)] $\{T(\gamma)\psi\}_{\gamma \in \Gamma}$ is a frame for $\psis{\psi}$ with constants $A$ and $B$.
\item[b)] For all $F \in L^2(\vn(\Gamma),[\psi,\psi])$
$$
A \tau(|F|^2\,[\psi,\psi]) \leq \tau\left(|F|^2 [\psi,\psi]^2\right) \leq B \tau(|F|^2\,[\psi,\psi])
$$
\item[c)] The following inequality holds
$$
A s_{[\psi,\psi]} \leq [\psi,\psi] \leq B s_{[\psi,\psi]} .
$$
\end{itemize}
$a) \Rightarrow b)$. Recall that $\{T(\gamma)\psi\}_{\gamma \in \Gamma} \subset \Hil$ is a frame for $\psis{\psi}$ with constants $A, B$ if
\begin{equation}\label{eq:framecondition}
A \|\varphi\|_\Hil^2 \leq \sum_{\gamma \in \Gamma} |\langle \varphi, T(\gamma)\psi\rangle_{\Hil}|^2 \leq B \|\varphi\|_\Hil^2
\end{equation}
for all $\varphi \in \psis{\psi}$. Let $F$ be in $L^2(\vn(\Gamma),[\psi,\psi])$. By Proposition \ref{prop:isometry}, there exists $\varphi_F \in \psis{\psi}$ such that $S_\psi[\varphi_F] = F$, and
\begin{equation}\label{eq:inproofFRisom}
\|F\|^2_{2,\psi} = \tau(|F|^2[\psi,\psi]) = \|\varphi_F\|_\Hil^2 .
\end{equation}
Since $S_\psi$ is an isometric isomorphism, we have
\begin{eqnarray*}
\sum_{\gamma \in \Gamma} |\langle \varphi_F, T(\gamma)\psi\rangle_\Hil|^2 & = & \sum_{\gamma \in \Gamma} |\langle S_\psi[\varphi_F], S_\psi[T(\gamma)\psi]\rangle_{2,\psi}|^2 = \sum_{\gamma \in \Gamma} |\langle F, \lr(\gamma)\rangle_{2,\psi}|^2\nonumber\\
& = & \sum_{\gamma \in \Gamma} |\tau(\lr(\gamma)^*F[\psi,\psi])|^2 = \sum_{\gamma \in \Gamma} |\widehat{F[\psi,\psi]}(\gamma)|^2
\end{eqnarray*}
where $\widehat{F[\psi,\psi]}(\gamma) := \tau(\lr(\gamma)^*F[\psi,\psi])$ are Fourier coefficients of an operator that is in $L^1(\vn(\Gamma))$. Indeed since $F \in L^2(\vn(\Gamma),[\psi,\psi])$, we have that $F[\psi,\psi]^{\frac12}$ is in $L^2(\vn(\Gamma))$, so by H\"older inequality $F[\psi,\psi] \in L^1(\vn(\Gamma))$, because $[\psi,\psi]^{\frac12} \in L^2(\vn(\Gamma))$ by definition.

By the assumption (\ref{eq:framecondition}) this series converges, so the sequence $\{\widehat{F[\psi,\psi]}(\gamma)\}_{\gamma \in \Gamma}$ is in $\HD$. Then, by Plancherel Lemma \ref{lem:Plancherel}, the traciality of $\tau$ and the positivity of $[\psi,\psi]$ we have
\begin{equation}\label{eq:inproofFRtrace}
\sum_{\gamma \in \Gamma} |\langle \varphi_F, T(\gamma)\psi\rangle_\Hil|^2 = \|F[\psi,\psi]\|_2^2 = \tau(|F|^2[\psi,\psi]^2) .
\end{equation}
Replacing now (\ref{eq:inproofFRisom}) and (\ref{eq:inproofFRtrace}) in (\ref{eq:framecondition}) we obtain the desired result $b)$.

\vspace{4pt}
$b) \Rightarrow a)$. Assume that $b)$ holds, and choose $\varphi \in \psis{\psi}$. 
Define $F_\varphi = S_\psi[\varphi] \in L^2(\vn(\Gamma),[\psi,\psi])$ as in Proposition \ref{prop:isometry}, so that
\begin{equation}\label{eq:inproofFRisom2}
\tau(|F_\varphi|^2[\psi,\psi]) = \|F_\varphi\|_{2,\psi}^2 = \|\varphi\|_\Hil^2 .
\end{equation}
By $b)$, we also have that $F_\varphi[\psi,\psi] \in L^2(\vn(\Gamma))$, since
$$
\|F_\varphi[\psi,\psi]\|_2^2 = \tau(|F_\varphi|^2[\psi,\psi]^2) \leq B \tau(|F_\varphi|^2[\psi,\psi]) = B \|F_\varphi\|_{2,\psi}^2 < \infty .
$$
By Plancherel Lemma \ref{lem:Plancherel} and the same computation leading to (\ref{eq:inproofFRtrace}), we get
\begin{equation}\label{eq:inproofFRtrace2}
\sum_{\gamma \in \Gamma} |\langle \varphi, T(\gamma)\psi\rangle_\Hil|^2 = \|F_\varphi[\psi,\psi]\|_2^2 = \tau(|F_\varphi|^2[\psi,\psi]^2) .
\end{equation}
The result then follows by substituting (\ref{eq:inproofFRisom2}) and (\ref{eq:inproofFRtrace2}) in $b)$.

\vspace{4pt}
$b) \Rightarrow c)$. The inequalities in $b)$ hold in particular for all $F \in \vsp\{\lr(\gamma)\}_{\gamma \in \Gamma}$, and for $F = \lr(\id) = \Id$ the right hand side implies that $[\psi,\psi]^2 \in L^1(\vn(\Gamma))$. We can then apply Lemma \ref{lem:Kaplansky} and obtain
$$
A[\psi,\psi] \leq [\psi,\psi]^2 \leq B [\psi,\psi]
$$
as $L^1(\vn(\Gamma))$ operators. Since $s_{[\psi,\psi]}$ and $[\psi,\psi]$ commute, the left inequality reads
\begin{eqnarray*}
0 & \leq & [\psi,\psi]^2 - A [\psi,\psi]\\
& = & s_{[\psi,\psi]}[\psi,\psi]^\frac12s_{[\psi,\psi]}\Big([\psi,\psi] - A s_{[\psi,\psi]}\Big)s_{[\psi,\psi]}[\psi,\psi]^\frac12s_{[\psi,\psi]}\, ,
\end{eqnarray*}
so if we call $[\psi,\psi]^2 - A [\psi,\psi] = x^*x$, and since the composition of $s_{[\psi,\psi]}[\psi,\psi]^\frac12s_{[\psi,\psi]}$ with $s_{[\psi,\psi]}[\psi,\psi]^{-\frac12}s_{[\psi,\psi]}$ is $s_{[\psi,\psi]}$, we obtain
$$
[\psi,\psi] - A s_{[\psi,\psi]} = \Big(x s_{[\psi,\psi]}[\psi,\psi]^{-\frac12}s_{[\psi,\psi]}\Big)^*\, \Big(x s_{[\psi,\psi]}[\psi,\psi]^{-\frac12}s_{[\psi,\psi]}\Big) \geq 0
$$
which proves the left inequality in $c)$. The right inequality is proved similarly starting with the right hand side inequality in $b)$.

\vspace{4pt}
$c) \Rightarrow b)$. Let us consider the right inequality in $c)$
$$
Bs_{[\psi,\psi]} - [\psi,\psi] \geq 0.
$$
and let us call it $x^*x$. Then for any $F \in L^2(\vn(\Gamma),[\psi,\psi])$, if we set $y = [\psi,\psi]^{\frac12} F^*$, we get
$$
0 \leq (xy)^*xy = F [\psi,\psi]^{\frac12} (B s_{[\psi,\psi]} - [\psi,\psi]) [\psi,\psi]^{\frac12} F^* .
$$
Since $[\psi,\psi]$ is positive and $[\psi,\psi]^{\frac12}$ commutes with $s_{[\psi,\psi]}$, by the traciality of $\tau$ this implies
$$
0 \leq \tau\Big( F [\psi,\psi]^{\frac12} (B s_{[\psi,\psi]} - [\psi,\psi]) [\psi,\psi]^{\frac12} F^* \Big) = \tau\Big( |F|^2 [\psi,\psi] (B s_{[\psi,\psi]} - [\psi,\psi]) \Big)
$$
which proves the right inequality in $b)$. The other inequality follows by a similar argument.
\end{proof}

\section{Dual integrability}

\subsection{A characterization}

This subsection is devoted to show that dual integrable representations are \emph{square integrable} in the sense that their matrix coefficients are in $\ell^2(\Gamma)$. Note that this notion of square integrability does not involve irreducibility, so that it is different from the one used e.g. in \cite{DufloMoore76, GrossmanMorletPaul86, CorwinGreenleaf90}. The formulation and the proof that we give follow the ideas of \cite[Corollary 3.4]{HSWW10}.

\begin{theo}\label{theo:squareintegrable}
Let $T$ be a unitary representation of the discrete countable group $\Gamma$ on a separable Hilbert space $\Hil$. Then the following are equivalent
\begin{itemize}
\item[i.] $T$ is dual integrable.
\item[ii.] $T$ is square integrable, in the sense that there exists a dense subspace $\mathcal{D}$ of $\Hil$ such that
$$
\Big\{\gamma \mapsto \langle \varphi , T(\gamma) \psi \rangle_{\Hil} \Big\}_{\gamma \in \Gamma} \in \HD \qquad \forall \ \psi \in \mathcal{D}, \ \forall \varphi \in \Hil.
$$
\item[iii.] $T$ is unitarily equivalent to a subrepresentation of a direct sum of countably many copies of the left regular representation $\lambda$.
\end{itemize}
\end{theo}
\begin{proof}
In order to start the proof we recall a basic construction for a unitary representation $T$ on a separable Hilbert space $\Hil$. By unitarity and the homomorphism property, we can construct a family $\Psi = \{\psi_i\}_{i \in \ind}$ for a countable set of indices $\ind$ such that
\begin{equation}\label{eq:GramSchmidt}
\Hil = \bigoplus_{i \in \ind} \psis{\psi_i}
\end{equation}
in the following way: choose $0 \neq \psi_1 \in \Hil$, set $V_1 := (\psis{\psi_1})^\bot$ and choose $\psi_2 \in V_1$. Then $\psis{\psi_2} \subset V_1$ because
$$
\psi_2 \in V_1 \ \Rightarrow \ \langle T(\gamma)\psi_2 , T(\gamma') \psi_1 \rangle_{\Hil} = \langle \psi_2 , T(\gamma^{-1}\gamma') \psi_1 \rangle_{\Hil} = 0 \quad \forall \gamma,\gamma' \in \Gamma \, ,
$$
so we can iterate this procedure and construct successive orthogonal subspaces.

In terms of such decomposition, for each $\varphi \in \Hil$ we can write
$$
\varphi = \sum_{i \in \ind} \varphi^i \ , \quad \varphi^i \in \psis{\psi_i} \quad \forall \ i \in \ind
$$
noting that $\|\varphi\|_\Hil^2 = \sum_{i \in \ind} \|\varphi^i\|_{\Hil}^2$.

$i. \Rightarrow ii.$ If $T$ is dual integrable, for the familily $\Psi = \{\psi_i\}_{i \in \ind}$ constructed above, we can define a map $S_\Psi$ on $\Hil$ as
$$
S_{\Psi}[\varphi] = \Big\{S_{\psi_i}[\varphi^i]\Big\}_{i \in \ind}
$$
where $S_{\psi_i}$ is the isometric isomorphism from $\psis{\psi_i}$ onto $L^2(\vn(\Gamma),[\psi_i,\psi_i])$ as in Proposition \ref{prop:isometry}. This implies that $S_{\Psi}[\varphi]$ is also an isometric isomorphism from $\Hil$ onto the Hilbert space $\mathcal{K}^T_\Psi$ given by\footnote{Observe that $\mathcal{K}^T_\Psi$ may depend on $T$ since its definition is based on the bracket map associated with $T$.}
$$
\mathcal{K}^T_\Psi = \bigoplus_{i \in \ind} L^2(\vn(\Gamma),[\psi_i,\psi_i])
$$
endowed with the natural norm. Indeed
$$
\|S_{\Psi}[\varphi]\|_{\mathcal{K}_\Psi}^2 = \sum_{i \in \ind} \|S_{\psi_i}[\varphi^i]\|_{L^2(\vn(\Gamma),[\psi_i,\psi_i])}^2 = \sum_{i \in \ind} \|\varphi^i\|_{\Hil}^2 = \|\varphi\|_\Hil^2 .
$$
Now let $\lr^i$ be the unitary representation on $L^2(\vn(\Gamma),[\psi_i,\psi_i])$ given by
$$
U(L^2(\vn(\Gamma),[\psi_i,\psi_i])) \ni \lr^i(\gamma) : F \mapsto \lr(\gamma)F \quad \forall \ \gamma \in \Gamma
$$
and denote with $\lr^{\Psi}$ the direct sum of these representations
$$
\lr^{\Psi} := \bigoplus_{i \in \ind} \lr^i : \Gamma \to U(\mathcal{K}^T_\Psi) .
$$
So, since
$$
S_{\Psi}[T(\gamma)\varphi] = \Big\{\lr^i(\gamma)S_{\psi_i}[\varphi^i]\Big\}_{i \in \ind} = \lr^{\Psi}(\gamma) S_{\Psi}[\varphi]
$$
the representation $T$ is unitarily equivalent to the representation $\lr^\Psi$.

Observe that each space $L^2(\vn(\Gamma),[\psi_i,\psi_i])$ is isometrically isomorphic to the Hilbert space of $L^2(\vn(\Gamma))$ operators supported by the bracket of $[\psi_i,\psi_i]$ and endowed with the $L^2(\vn(\Gamma))$ norm
$$
L^2(\vn(\Gamma))s_{[\psi_i,\psi_i]} = \{F \in L^2(\vn(\Gamma)) \, |\, \exists G \in L^2(\vn(\Gamma)) : F = Gs_{[\psi_i,\psi_i]}\}
$$
via the map
$$
\begin{array}{rccl}
B_{\psi_i} : & L^2(\vn(\Gamma),[\psi_i,\psi_i]) & \to & L^2(\vn(\Gamma))s_{[\psi_i,\psi_i]}\\
& F & \mapsto & F [\psi_i,\psi_i]^\frac12 .
\end{array}
$$
In the notation of Lemma \ref{lem:weighted} we have that $L^2(\vn(\Gamma))s_{[\psi_i,\psi_i]} = \ol{\,\mathfrak{h}\ }^{\|\cdot\|_{2}}$, so it is a closed subspace of $L^2(\vn(\Gamma))$. By Plancherel Lemma \ref{lem:Plancherel}, the Fourier transform maps $L^2(\vn(\Gamma))$ isometrically onto $\ell_2(\Gamma)$, so it also maps $L^2(\vn(\Gamma))s_{[\psi_i,\psi_i]}$ isometrically onto a closed subspace of $\ell_2(\Gamma)$. We have then that each $\lambda^i$ is unitarily equivalent, via the Fourier transform, to a subrepresentation of the left regular representation and, as such, it is square integrable (for a proof\footnote{In this type of statements one usually encounters the hypothesis of irreducibility of the representation. Note however that this is not required to prove the considered implication.} see e.g. \cite[Corollary 12.2.4]{DeitmarEchterhoff09}). The dual integrable representation $T$ is then unitarily equivalent to a direct sum of square integrable representations, and hence it is square integrable.

$ii. \Rightarrow iii.$ If $T$ is square integrable, then the operator $\mathcal{A}_\psi$ defined by
$$
\mathcal{A}_\psi \varphi(\gamma) = \langle \varphi, T(\gamma)\psi \rangle_{\Hil}
$$
and sometimes called the \emph{analysis operator}, is a well defined operator from $\Hil$ into $\HD$ for all $\psi \in \mathcal{D}$. This is a condition generally referred to as the fact that $\{T(\gamma)\psi\}_{\gamma \in \Gamma}$ is a \emph{Bessel sequence} (see e.g. \cite{HW96}), which implies that the operator $\Frame_\psi := \mathcal{A}_\psi^*\mathcal{A}_\psi$ is well defined and bounded from $\Hil$ into $\psis{\psi}$. 
This is sometimes called the \emph{frame operator}, and its action reads explicitly
$$
\Frame_\psi \varphi = \sum_{\gamma \in \Gamma} \mathcal{A}_\psi \varphi (\gamma) \, T(\gamma)\psi .
$$
This allows to write the polar decomposition of $\mathcal{A}_\psi$ as
$$
\mathcal{A}_\psi = u_{\mathcal{A}_\psi} |\mathcal{A}_\psi| = u_{\mathcal{A}_\psi} \Frame_\psi^\frac12
$$
where $u_{\mathcal{A}_\psi}$ is a partial isometry from $\Ker(\mathcal{A}_\psi)^\bot = \psis{\psi}$ to $\HD$ satisfying $u_{\mathcal{A}_\psi}^* u_{\mathcal{A}_\psi}^{\phantom{*}} = r_{\mathcal{A}_\psi}^{\phantom{*}} = \Proj_{\psis{\psi}}$.

\newpage
Since $T$ is a unitary representation, $\mathcal{A}_\psi$ satisfies $\mathcal{A}_\psi T(\gamma) = \lr(\gamma)\mathcal{A}_\psi$, which implies that $\Frame_\psi$ commutes with $T(\gamma)$ for all $\gamma \in \Gamma$, because
$$
\Frame_\psi T(\gamma)\varphi = \sum_{\gamma' \in \Gamma} (\mathcal{A}_\psi T(\gamma)\varphi) (\gamma') \, T(\gamma')\psi = \sum_{\gamma' \in \Gamma} \mathcal{A}_\psi \varphi (\gamma^{-1}\gamma') \, T(\gamma')\psi = T(\gamma) \Frame_\psi \varphi .
$$
From such commutation relations and the polar decomposition of $\mathcal{A}_\psi$ we get
$$
\lr(\gamma) u_{\mathcal{A}_\psi} \Frame_\psi^\frac12 = u_{\mathcal{A}_\psi} \Frame_\psi^\frac12 T(\gamma) = u_{\mathcal{A}_\psi}  T(\gamma) \Frame_\psi^\frac12 \, ,
$$
so $(\lr(\gamma) u_{\mathcal{A}_\psi} - u_{\mathcal{A}_\psi}  T(\gamma))\Frame_\psi^\frac12 = 0$, which implies $(\lr(\gamma) u_{\mathcal{A}_\psi} - u_{\mathcal{A}_\psi}  T(\gamma))s_{|\mathcal{A}_\psi|} = 0$, where $s_{|\mathcal{A}_\psi|}$ is the support of $|\mathcal{A}_\psi| = \Frame_\psi^\frac12$. Since $s_{|\mathcal{A}_\psi|}$ coincides with $u_{\mathcal{A}_\psi}^* u_{\mathcal{A}_\psi}^{\phantom{*}} = r_{\mathcal{A}_\psi}^{\phantom{*}}$, the right support of $\mathcal{A}_\psi$, and since $T$ commutes with $s_{|\mathcal{A}_\psi|}$ because it commutes with $\Frame_\psi$, we then have that
\begin{equation}\label{eq:commutation}
\lr(\gamma) u_{\mathcal{A}_\psi} = u_{\mathcal{A}_\psi}  T(\gamma) \quad \forall \ \gamma \in \Gamma .
\end{equation}
Let now $\Psi = \{\psi_i\}_{i \in \ind} \subset \mathcal{D}$ be a countable family such that the decomposition (\ref{eq:GramSchmidt}) holds, and define the map
$U_\Psi : \Hil \to \ell_2(\ind,\HD)$ by
$$
U_\Psi\varphi = \{u_{\mathcal{A}_{\psi_i}}\varphi^i\}_{i \in \ind} .
$$
Since each $u_{\mathcal{A}_{\psi_i}}$ is a partial isometry from $\psis{\psi_i}$ into $\HD$, then $U_\Psi$ is an isometry, and by (\ref{eq:commutation}), $T$ is unitarily equivalent to a subrepresentation of a direct sum of countably many copies of the left regular representation.


$iii. \Rightarrow i.$ First of all, since by Lemma \ref{lem:dualint} the left regular representation is dual integrable, then a direct sum $\lambda^\ind = \displaystyle \bigoplus_{i \in \ind} \lambda$ of countably many copies of $\lambda$, where $\ind$ is a countable set, is also dual integrable. Indeed, for $\varphi = \{\varphi^i\}_{i \in \ind}$ and $\psi = \{\psi^i\}_{i \in \ind}$ in $\ell_2(\ind,\HD)$ we have
\begin{align*}
\langle \varphi, \lambda^\ind(\gamma)\psi \rangle_{\ell_2(\ind,\HD)} & = \sum_{i \in \ind} \langle \varphi^i, \lambda(\gamma) \psi^i\rangle_{\HD} = \sum_{i \in \ind} \tau\left(\lambda^*(\gamma)\big[\varphi^i, \psi^i\big]\right)\\
& = \tau\left(\lambda^*(\gamma)\sum_{i \in \ind} \big[\varphi^i, \psi^i\big]\right)
\end{align*}
where the last identity is due to Fubini's theorem. Since, again by Lemma \ref{lem:dualint}, $[\varphi^i, \psi^i\big] = \Phi_i^{\phantom{*}} \Psi_i^*$, with $\wh{\Phi_i} = \varphi^i$ and $\wh{\Psi_i} = \psi^i$ for all $i \in \ind$, then
$$
\sum_{i \in \ind} \big[\varphi^i, \psi^i\big] = \sum_{i \in \ind} \Phi_i^{\phantom{*}} \Psi_i^* .
$$
When $\varphi = \psi$, we can see that this belongs to $L^1(\vn(\Gamma))$ by Plancherel theorem because $\tau\left(\left|\sum_{i \in \ind} \Phi_i^{\phantom{*}} \Phi_i^*\right|\right) = \|\varphi\|^2_{\ell_2(\ind,\HD)}$, while the general case can be obtained by polarization, so this sum defines a bracket map.

Secondly, observe that if a unitary representation is dual integrable then all its subrepresentations are dual integrable, with the same bracket map restricted to the closed Hilbert subspace. Finally, if a unitary representation $T_1$ over the Hilbert space $\Hil_1$ is unitarily equivalent to a dual integrable representation $T_2$ over the Hilbert space $\Hil_2$ via an isometric isomorphism $U : \Hil_1 \to \Hil_2$, then $T_1$ is dual integrable with bracket $[\varphi, \psi]_{T_1} = [U\varphi, U\psi]_{T_2}$ where $\varphi, \psi \in \Hil_1$.
\end{proof}

\begin{rem}
If $T$ is a square/dual integrable representation of a discrete countable group $\Gamma$ on the separable Hilbert space $\Hil$, then for all $\varphi \in \Hil$ and all $\psi \in \mathcal{D}$ as in Theorem \ref{theo:squareintegrable} the operator
$$
B_{\varphi,\psi} := \sum_{\gamma \in \Gamma} \langle \varphi , T(\gamma) \psi \rangle_{\Hil} \lr(\gamma)
$$
belongs to $L^2(\vn(\Gamma))$, and by Plancherel Lemma \ref{lem:Plancherel}
$$
\langle \varphi , T(\gamma) \psi \rangle_{\Hil} = \tau(\lr(\gamma)^*B_{\varphi,\psi})
$$
so that it coincides with the bracket map $[\varphi,\psi]$ for all $\varphi \in \Hil$ and all $\psi \in \mathcal{D}$.
\end{rem}

\subsection{A noncommutative Zak transform} \label{sec:NCZaktransform}

Following the definition of Zak transform given for LCA groups in \cite{HSWW10}, we can obtain dual integrability for unitary representations obtained from certain actions of discrete countable groups on measure spaces in terms of a properly defined noncommutative Zak transform.

Let $(\meas,\mu)$ be a $\sigma$-finite measure space. A measurable action of a group $\Gamma$ on $\meas$ is a map $\act: \Gamma  \times \meas \to \meas$ satisfying
\begin{itemize}
\item[i.] for each $\gamma \in \Gamma$ the map $\act_\gamma: \meas \to \meas$  given by $\act_\gamma (x) :=\act(\gamma,x)$ is $\mu$-measurable;
\item[ii.] $\act_\gamma (\act_{\gamma'} (x)) = \act_{\gamma \gamma'}(x)$ for all $\gamma, \gamma' \in \Gamma$ and all $x \in \meas$;
\item[iii.] $\act_\id (x) = x$ for all $x \in \meas$.
\end{itemize}
For each $\gamma \in \Gamma$, let us denote with $\mu_\gamma$ the measure defined by
$$
\mu_\gamma(E) := \mu(\act_\gamma (E))
$$
for all Borel sets $E \subset \meas$. We say that the action $\act$ is \emph{quasi-$\Gamma$-invariant} if $\mu_\gamma \ll \mu$ with positive Radon-Nikodym derivative for all $\gamma \in \Gamma$, i.e. if there exists a measurable function $J_\act: \Gamma \times \meas \to \R^+$, called Jacobian of $\act$, such that
\begin{equation}\label{Jacobian}
d\mu(\act_\gamma (x)) = J_\sigma (\gamma, x)\, d\mu (x)\,.
\end{equation}
Using the properties of $\act$ and the definition of the Jacobian given in (\ref{Jacobian}) we deduce:
\begin{equation}\label{Chain-rule}
J_\act (\gamma_1\gamma_2 ,x) = J_\act(\gamma_1, \act_{\gamma_2} (x))\, J_\act(\gamma_2, x) \quad \forall \ \gamma_1,\gamma_2 \in \Gamma \, , \
\forall \ x \in \meas \,.
\end{equation}
For a quasi-$\Gamma$-invariant action $\sigma$ we can associate a unitary representation $T_\act$ of $\Gamma$ on $L^2(\meas,\mu)$ given by
\begin{equation}\label{eq:repZak}
T_\act(\gamma)f(x) := J_\act (\gamma^{-1}, x)^\frac12 f(\act_{\gamma^{-1}}(x)) .
\end{equation}
The fact that $T$ is unitary follows from (\ref{Jacobian}). To prove that it is a representation, choose $\gamma_1, \gamma_2 \in \Gamma,
f \in L^2(\meas,\mu)$ and $x \in \meas$, use definition (\ref{eq:repZak}), equality (\ref{Chain-rule}), and the properties of the
action $\sigma$ to obtain:
\begin{align*}
 (T_\act (\gamma_1) T_\act (\gamma_2) f) (x) & = J_\act (\gamma_1^{-1}, x)^{\frac12}(T_\act (\gamma_2) f)(\act_{\gamma_1^{-1}} (x)) \\
& = J_\act (\gamma_1^{-1}, x)^{\frac12} J_\act (\gamma_2^{-1}, \act_{\gamma_1^{-1}}(x))^{\frac12} f(\act_{\gamma_2^{-1}}(\act_{\gamma_1^{-1}}(x)) \\
& = J_\act (\gamma_2^{-1} \gamma_1^{-1}, x)^{\frac12} f(\act_{\gamma_2^{-1} \gamma_1^{-1}}(x)) \\
& = J_\act ((\gamma_1 \gamma_2)^{-1}, x)^{\frac12} f(\act_{(\gamma_1 \gamma_2)^{-1}}(x)) = (T_\act(\gamma_1\gamma_2)f) (x)\,.
\end{align*}

The \emph{Zak transform} of $\psi \in L^2(\meas,\mu)$ associated to the action $\act$ is the field over $\meas$ whose values are operators given by
\begin{equation}\label{eq:generalizedZak}
Z_\act[\psi](x) := \sum_{\gamma \in \Gamma} \Big(\big(T_\act(\gamma)\psi\big)(x)\Big) \lr(\gamma)^*, \quad x \in \meas .
\end{equation}

We will say that the action $\act$ has the \emph{tiling property} if there exists a $\mu$-measurable subset $C \subset \meas$ such that the family
$\{\act_\gamma (C)\}_{\gamma \in \Gamma}$ is a disjoint covering of $\mu$-almost all $\meas$, i.e.
$\mu\big(\act_{\gamma_1}(C) \cap \act_{\gamma_2} (C)\big) = 0$ for $\gamma_1 \neq \gamma_2$ and
$$
\mu\bigg(\meas \setminus \bigcup_{\gamma \in \Gamma}\,\act_\gamma (C)\bigg) = 0 .
$$

The Zak transform given by (\ref{eq:generalizedZak}) is defined on $\meas$. We now show that if the action $\act$ satisfies
the tiling property it is enough to know its values in the tiling set $C$ to know its values in $\meas.$
This follows from the the following property:

\begin{lem}\label{lem:Zak1}
Given a quasi-$\Gamma$-invariant action $\sigma$ we have
$$ Z_\act [\psi] (\act_\gamma (x)) = J_\act (\gamma, x)^{-\frac12} Z_\act [\psi](x) \lr(\gamma)^*$$
for $\gamma\in \Gamma, x \in \meas$, and $\psi \in L^2(\meas,\mu)$.
\end{lem}
\begin{proof}
Using the definitions of $Z_\act$ and $T_\act$, and a change of variables we obtain:
\begin{align*}
Z_\act [\psi] (\act_\gamma (x)) & = \sum_{\gamma_1 \in \Gamma} J_\act (\gamma_1^{-1}, \act_\gamma(x))^{\frac12} \,\psi(\act_{\gamma_1^{-1}\gamma}(x))\, \lr(\gamma_1)^* \\
& = \sum_{\gamma_2 \in \Gamma} J_\act ((\gamma \gamma_2)^{-1} , \act_\gamma (x))^{\frac12} \, \psi (\act_{\gamma_2^{-1}}(x)) \, \lr(\gamma \gamma_2)^* \,.
\end{align*}
Using the Chain rule (\ref{Chain-rule}) we deduce:
\begin{align*}
 Z_\act [\psi] (\act_\gamma (x)) & = \sum_{\gamma_2 \in \Gamma} J_\act (\gamma_2^{-1}, x)^{\frac12} J_\act (\gamma^{-1}, \act_{\gamma}(x))^{\frac12}
 \,\psi(\act_{\gamma_2^{-1}}(x))\, \lr(\gamma_2)^*\lr(\gamma)^* \\
& = J_\act (\gamma^{-1}, \act_{\gamma}(x))^{\frac12} \Big( \sum_{\gamma_2 \in \Gamma} J_\act (\gamma_2^{-1}, x)^{\frac12}
 \,\psi(\act_{\gamma_2^{-1}}(x))\, \lr(\gamma_2)^*\Big) \lr(\gamma)^* \\
& = J_\act (\gamma^{-1}, \act_{\gamma}(x))^{\frac12}\, Z_\act [\psi](x) \,\lr(\gamma)^*\,.
\end{align*}
It remains to show that $J_\act (\gamma^{-1}, \act_{\gamma}(x))= J_\act (\gamma, x)^{-1}.$ This follows from the definition of
the Jacobian (\ref{Jacobian}) since
$$
J_\act(\gamma^{-1}, \act_\gamma(x)) d\mu (\act_\gamma(x)) = d\mu(\act_{\gamma^{-1}}(\act_{\gamma}(x))) = d\mu (x)
= J_\act(\gamma, x)^{-1} d\mu (\act_\gamma(x))\,.
$$
\end{proof}

We are now ready to prove Theorem B stated in the Introduction, showing that when a quasi-$\Gamma$-invariant action has the tiling property, then the associated unitary representation (\ref{eq:repZak}) is dual integrable with a bracket map which can be expressed in terms of the Zak transform (\ref{eq:generalizedZak}).

\begin{proof}[Proof of Theorem B]
The quasi-periodicity is obtained from (\ref{eq:generalizedZak}) and the fact that $T_\act$ is a representation of $\Gamma$ since
\begin{align*}
Z_\act[T_\act(\gamma)\psi](x) & = \sum_{\gamma' \in \Gamma} \Big(\big(T_\act(\gamma'\gamma)\psi\big)(x)\Big) \lr(\gamma'^{-1})\\
& = \sum_{\gamma'' \in \Gamma} \Big(\big(T_\act(\gamma'')\psi\big)(x)\Big) \lr(\gamma)\lr(\gamma''^{-1}) = \lr(\gamma) \,Z_\act [\psi](x)\,.
\end{align*}
The isometry property of $Z_\act$ and the dual integrability of $T_\act$ will follow from the computations below. For $\gamma \in \Gamma$ and $\psi_1, \psi_2 \in L^2(\meas,\mu)$, using Fubini's theorem, the traciality of $\tau$ and the quasi-periodicity of $Z_\act$, we get
\begin{align*}
\tau&\bigg(\lr(\gamma)^* \!\int_C\!\! Z_\act[\psi_1](x) Z_\act[\psi_2](x)^* d\mu(x)\bigg)
\!=\! \int_C \!\!\tau\Big(\lr(\gamma)^* Z_\act[\psi_1](x)Z_\act[\psi_2](x)^* \Big) d\mu(x)\\
& = \int_C \tau\Big(\,\big(\lr(\gamma)Z_\act[\psi_2](x)\big)^* Z_\act[\psi_1](x) \Big) d\mu(x)\\
& = \int_C \tau\Big(Z_\act[T_\act(\gamma)\psi_2](x)^* Z_\act[\psi_1](x) \Big) d\mu(x)\\
& = \int_C \langle Z_\act[\psi_1](x) \delta_\id, Z_\act[T_\act(\gamma)\psi_2](x) \delta_\id\rangle_{\HD}d\mu(x)\,,
\end{align*}
so by the definition of $Z_\act$, the orthonormality of $\{\delta_\gamma\}_{\gamma \in \Gamma}$ and the tiling property
\begin{align*}
\phantom{\tau} & = \int_C\! \langle \sum_{\gamma_1 \in \Gamma} \Big(\big(T_\act(\gamma_1)\psi_1\big)(x)\Big) \delta_{\gamma_1^{-1}}, \sum_{\gamma_2 \in \Gamma} \Big(\big(T_\act(\gamma_2)(T_\act(\gamma)\psi_2)\big)(x)\Big) \delta_{\gamma_2^{-1}}\rangle_{\HD}d\mu(x)\\
& = \int_C \sum_{\gamma_1 \in \Gamma} \Big(\big(T_\act(\gamma_1)\psi_1\big)(x)\Big) \ol{\Big(\big(T_\act(\gamma_1)(T_\act(\gamma)\psi_2)\big)(x)\Big)} d\mu(x)\\
& = \sum_{\gamma_1\in \Gamma} \int_C J_\act (\gamma_1^{-1}, x) \,\psi_1 (\act_{\gamma_1^{-1}}(x)) \overline{(T_\act(\gamma)\psi_2) (\act_{\gamma_1^{-1}}(x))}\, d\mu(x)\,.
\end{align*}
Letting $y= \act_{\gamma_1^{-1}}(x)$ we deduce from (\ref{Jacobian}) the equalities
$$
d\mu(y)= d\mu (\act_{\gamma_1^{-1}}(x) ) = J_\act (\gamma_1^{-1}, x) d\mu(x)\,,
$$
so that
\begin{align*}
\tau\bigg(\lr(\gamma)^* & \int_C Z_\act[\psi_1](x) Z_\act[\psi_2](x)^* d\mu(x)\bigg) \\
& = \sum_{\gamma_1 \in \Gamma} \int_{\act_{\gamma_1^{-1}}(C)} \psi_1(y) \, \overline{T_\act(\gamma) \psi_2(y)} \, d\mu(y)
 =\langle \psi_1, T_\act(\gamma)\psi_2\rangle_{L^2(\meas,\mu)} .
\end{align*}
The isometry property can be obtained choosing $\psi_1 = \psi_2 = \psi \in L^2(\meas,\mu)$ and $\gamma = \id$, using the traciality of $\tau$, since
$$
\int_C \|Z_\act[\psi](x)\|_2^2 d\mu(x) = \int_C \tau\Big(Z_\act[\psi](x) Z_\act[\psi](x)^*\Big)d\mu(x) = \|\psi\|^2_{L^2(\meas,\mu)}\, ,
$$
while the fact that the bracket (\ref{eq:ZakBracket}) is in $L^1(\vn(\Gamma))$ is a consequence of H\"older inequality on $L^2(\vn{\Gamma})$ and on $L^2(C, d\mu(x))$:
\begin{align*}
\bigg\| \int_C Z_\act[\psi_1](x) & Z_\act[\psi_2](x)^* d\mu(x) \bigg\|_{L^1(\vn{\Gamma})} \leq \int_C \tau\left(|Z_\act[\psi_1](x) Z_\act[\psi_2](x)^*|\right)d\mu(x)\\
& \leq \int_C \|Z_\act[\psi_1](x)\|_2 \|Z_\act[\psi_2](x)\|_2 d\mu(x)\\
& \leq \left(\int_C \|Z_\act[\psi_1](x)\|_2^2 d\mu(x)\right)^\frac12 \left(\int_C \|Z_\act[\psi_2](x)\|_2^2 d\mu(x)\right)^\frac12\\
& = \|\psi_1\|_{L^2(\meas,\mu)} \|\psi_2\|_{L^2(\meas,\mu)} .\qedhere
\end{align*}
\end{proof}

Conditions on $(\meas,\mu)$, $\Gamma$ and $\sigma$ that are equivalent to the tiling property can be stated in terms of the existence of a Borel transversal (set of representatives) for the space of orbits of $\Gamma$ in $\meas$, and can be found e.g. in \cite{Baker65, Effros65}.

\section{Examples} \label{Sec:Examples}

In this section we will start by showing that for LCA groups our results imply the ones obtained in \cite{HSWW10}. Then we will consider
the case of the left action of a discrete subgroup, and that of semidirect products of groups with associated unitary representations, including the ones used in the theory of composite wavelets (see e.g. \cite{WeissWilson01, LaugersenWeaverWeissWilson02, LabateWeiss10}).

\subsection{LCA groups} \label{Subsec:LCAgroups}

In this paper we have introduced a bracket map whose values are operators belonging to $L^1(\vn (\Gamma))$, where $\Gamma$ is any
discrete group. On the other hand, when $\Gamma$ is a LCA group, a bracket map has been defined
in \cite{HSWW10} that takes values on $L^1(\widehat \Gamma)$, where $\widehat \Gamma$ denotes the Pontryagin dual of $\Gamma.$
For a  discrete LCA group we, therefore, have two definitions of a bracket map of different nature: the one defined in this
paper, which is operator valued, denoted by $[ \ , \ ]^{\textnormal{op}}$ in this subsection, and the one defined in \cite{HSWW10}, whose values are functions, and denoted by $[ \ , \ ]$ in this subsection.

The relation between the operator bracket $[ \ , \ ]^{\textnormal{op}}$ and the function bracket $[ \ , \ ]$, as well as the equivalence of the corresponding results when $\Gamma$ is a countable abelian group, are due to the following standard construction which corresponds to a special case of the Gelfand-Naimark-Segal (GNS) theorem (see e.g. \cite[\S 7]{Conway00}, \cite[Vol.1, Chapt. 4]{KadisonRingrose83}).
For $\gamma \in \Gamma$, let $\chi_\gamma : \wh{\Gamma} \to \T$ be the associated character, and define the multiplication operator $\Lambda_{\chi_\gamma}$ on $L^2(\wh{\Gamma})$ as $\Lambda_{\chi_\gamma} \varphi = \chi_\gamma \varphi$ for $\varphi \in L^2(\wh{\Gamma})$. Then the map
$$
\Lambda : L^\infty(\wh{\Gamma}) \ni \chi_\gamma \mapsto \Lambda_{\chi_\gamma} \in \B(L^2(\wh{\Gamma}))
$$
is a $*$-homomorphism, i.e. a Banach algebra isomorphism that preserves the involution given by complex conjugation. This means that we may represent the von Neumann algebra $L^\infty(\wh{\Gamma})$ as the Banach subalgebra of $\B(L^2(\wh{\Gamma}))$ generated by the family of operators $\{\Lambda_{\chi_\gamma}\}_{\gamma \in \Gamma}$. We note that this implies in particular that the operator analog of characteristic functions of measurable sets are orthogonal projections. On the other hand, if 
$$
\F : \ell_2(\Gamma) \ni \delta_\gamma \mapsto \chi_\gamma \in L^2(\wh{\Gamma})
$$
denotes the classical Fourier transform, it is easily checked that $\lambda(\gamma) = \F^{-1}\Lambda_{\chi_\gamma}\F$.
Since $\F$ is a unitary isomorphism, we can deduce the well known fact that $L^\infty(\wh{\Gamma}) = \vn(\Gamma)$ in the sense that they are $*$-homomorphic via the map $\Lambda$ and the conjugation with the Fourier transform. This provides the relation
$$
[\varphi,\psi]^\textnormal{op} = \F^{-1} \Lambda_{[\varphi,\psi]} \F
$$
and allows to obtain the results of \cite{HSWW10} on bounds of the function $[\psi,\psi]$ in terms of the operator bounds of Theorem A. 

We refer to \cite{KadisonRingrose83, PisierXu03} for a more detailed exposition of abelian von Neumann algebras and their $L^p$ spaces in the terminology of the noncommutative theory.

\subsection{Left action of a discrete subgroup} \label{Subssec:leftaction}

Let $\G$ be a locally compact group, and $\Gamma$ a discrete countable subgroup of $\G$ acting from the left on $\G$ by
$$
L_\gamma (g) := \gamma g .
$$
The map $L: \Gamma \times \G \to \G$ given by $L(\gamma,g) = L_\gamma (g)$ is an action as defined in Section \ref{sec:NCZaktransform},
with Jacobian identically 1, considering in $\G$ its left-invariant Haar measure $\mu$. Therefore, we can consider the operator
$T_L$ as in (\ref{eq:repZak}) given by
\begin{equation} \label{eq:4.1.1}
T_L(\gamma)f (g) := f(\gamma^{-1} g)\,, \quad g\in \G\,, \ f\in L^2(\G, d\mu)\,.
\end{equation}
We can also consider the Zak transform associated to $L$, as in (\ref{eq:generalizedZak}), which is given by
\begin{equation}  \label{eq:4.1.2}
Z_L[\psi](g) := \sum_{\gamma\in \Gamma} \psi(\gamma^{-1}g) \lr(\gamma)^*\,, \quad g\in \G\,, \ \psi\in L^2(\G, d\mu)\,.
\end{equation}
In order to be able to apply Theorem B in this situation we need to make sure that $L$ has a measurable tiling
set.

\

Given a group $\G$ and a subgroup $M$ of $\G$, a \emph{transversal} of $\G$ by the right cosets of $M \backslash \G=
\{Mg : g \in \G\}$ is a subset $C$ of $\G$ such that $C\cap Mg$ is a set with only one element for each $g\in \G$ (i.e $C$ has
one and only one representative of each class). It is clear that to each transversal $C$ there corresponds a unique \emph{cross-section}
$q: M \backslash \G \rightarrow C$ such that  $\pi \circ q = Id_{M\backslash \G},$ where $\pi : \G \rightarrow M\backslash \G$ is the
canonical quotient map, $\pi(g)=Mg.$

Several papers in the 1960's were devoted to find conditions on $\G$ and $M$, so that a transversal $C$ can be chosen to be
measurable. It is proven in \cite{FelGree68} that if $\G$ is a locally compact group and $M$ is a closed metrizable subgroup, then
there is a measurable transversal $C$ for the left cosets of $M$. Since the proof also works for right cosets, in our situation we have
existence of a measurable transversal $C$. It turns out that this transversal $C$ is a tiling set, that is
$$
\G = \bigcup_{\gamma \in \Gamma} \gamma C\,, \quad \gamma_1 C \cap \gamma_2 C = \emptyset\, \ \mbox{if} \ \gamma_1 \neq \gamma_2\,.
$$
It is clear that $\bigcup_{\gamma \in \Gamma} \gamma C\ \subset \G$. It is worthwhile to write the proof of
$\G\subset \bigcup_{\gamma \in \Gamma} \gamma C $ to exhibit the reason why we need to take a transversal of the right cosets of $\G$ by
$\Gamma$ and not the left cosets. If $g \in \G$, consider $\Gamma g \in \Gamma \backslash \G$ and let $c = q(\Gamma g) \in C,$ where
$q$ is the cross-section associated to $C$. Then $\Gamma g = \pi \circ q (\Gamma g) = \pi(c) = \Gamma c.$ Thus, $g\in \Gamma c$, which
implies $g = \gamma c$ for some $\gamma \in \Gamma.$ Thus $g\in \gamma C\,.$ The disjointness property follows from the fact that
$C$ is a transversal of $\G$ by the right cosets of $\Gamma$.

Thus, we have the following corollary to Theorem B.

\begin{cor}\label{cor:dualintG}
Let $\G$ be a locally compact group with the left Haar measure $\mu$ and
$\Gamma$ a discrete countable subgroup of $\G$. Then the unitary representation $T_L$ on $L^2(\G)$ associated to the left action
 $L$ of $\Gamma$ on $\G$ given by (\ref{eq:4.1.1}) is dual integrable and the bracket map is given by
\begin{equation}\label{eq:bracketG}
[\psi_1, \psi_2] = \int_{C} Z_L[\psi_1](g)Z_L[\psi_2](g)^* d\mu(g)\,, \quad \psi_1, \psi_2 \in L^2(\G)\,.
\end{equation}
where $Z_L$ is the Zak transform (\ref{eq:4.1.2}) and $C$ is a measurable transversal.
\end{cor}

\begin{rem}
One could prove the dual integrability of the unitary representation $T_L$ given by (\ref{eq:4.1.1}) without using Theorem B.
It is a consequence of a classical result proved in \cite[Chapt. II, \S 9]{Weil65} (see also e.g. \cite[Th. 2.49]{Folland95}) that, in our situation, there exist a right quasi-invariant Radon measure $\nu$ on $\Gamma \backslash \G$ satisfying
$$
\nu(E g) = \Delta(g) \nu(E) \quad \forall \ E \in \mathcal{B}(\Gamma \backslash \G)
$$
where $\Delta: \G \to \R^+$ is the modular function of $(\G,\mu)$, such that for all $\psi \in \mathcal{C}_c(\G)$
the following periodization formula holds:
\begin{equation}\label{eq:WeilPeriodization}
\int_\G \psi(g) d\mu(g) =  \int_{\Gamma \backslash \G} \sum_{\gamma \in \Gamma} \psi(\gamma^{-1} g) d\nu(\Gamma g) .
\end{equation}
This periodization formula is all we need to prove that the unitary representation $T_L$ on $L^2(\G)$ given by (\ref{eq:4.1.1}) is
dual integrable. Moreover, formula (\ref{eq:bracketG}) holds with $C$ replaced by $\Gamma \backslash \G$ and $d\mu(g)$ replaced
by $d\nu(\Gamma g).$
\end{rem}

For integer translations $\Z^n$ over $\Rn$, this approach of a periodization in the spatial variables instead of the Fourier variables provides a different way of writing the bracket with respect to (\ref{eq:integertranslates}). The previous computation indeed reads
\begin{eqnarray*}
\langle \psi_1, T_k \psi_2 \rangle_{L^2(\Rn)} & = & \int_\Rn \psi_1(x) \ol{\psi_2(x - k)} dx\\
& = & \int_\Tn \sum_{m \in \Zn} \psi_1(m + h) \ol{\psi_2(m + h - k)} dh\\
& = & \int_\Tn \sum_{m \in \Zn} \psi_1^h(m) \ol{\psi_2^h(m - k)} dh = \int_\Tn \langle \psi_1^h, T_k \psi_2^h\rangle_{\ell_2(\Zn)} dh
\end{eqnarray*}
where $\psi^h(x) = \psi(x + h)$. Since by Plancherel theorem on $\Zn$
$$
\langle \psi_1^h, T_k \psi_2^h\rangle_{\ell_2(\Zn)} = \int_\Tn \F_\Zn\psi_1^h(\alpha) \ol{\F_\Zn\psi_2^h(\alpha)} e^{2\pi i \alpha k} d\alpha
$$
where
$$
\F_\Zn\psi^h(\alpha) := \sum_{k_1 \in \Zn} \psi(k_1+h) e^{- 2 \pi i k_1 \alpha} = Z\psi(h,\alpha)
$$
and $Z\psi$ is the ordinary Zak transform, we get
\begin{eqnarray*}
\langle \psi_1, T_k \psi_2 \rangle_{L^2(\Rn)} & = & \int_\Tn \int_\Tn Z\psi_1(h,\alpha) \ol{Z\psi_2(h,\alpha)} e^{2\pi i \alpha k}   d\alpha dh\\
& := & \int_\Tn \, [\psi_1, \psi_2](\alpha) e^{2\pi i \alpha k} \,d\alpha\,.
\end{eqnarray*}
An explicit expression obtained with this approach for the bracket map is then
\begin{equation}\label{eq:integertranslates2}
[\psi_1 , \psi_2](\alpha) = \int_\Tn \, Z\psi_1(h,\alpha) \, \ol{Z\psi_2(h,\alpha)} \, dh
\end{equation}
which coincides with (\ref{eq:integertranslates}) by $L^1$ uniqueness theorem of Fourier coefficients.

\subsection{Semidirect products} \label{Subsec:semidirecproducts}

Let $A$ be a locally compact group with left Haar measure $d\varsigma(x)$, $\G$ a locally compact group with left Haar measure $d\mu(g)$, and
$\act: \G \to \textnormal{Aut}(A)$ a measurable action of $\G$ on $A$. The semidirect product $S := A \ract \G$ is the topological group
endowed with the product topology and characterized by the composition law
\begin{equation}\label{eq:semidirectcomp}
s \bullet s' = (x,g)\bullet(x',g') = (x\,\act_g(x'), gg')
\end{equation}
Let us denote with $\varsigma_g$ the measure on $A$ given by
$$
\varsigma_g(E) = \varsigma(\act_g(E)) \ \ \forall \ E \in \mathcal{B}(A)
$$
and assume this measure is absolutely continuous with respect to $\varsigma$ with positive Radon-Nikodym derivative
$$
J_\act(g) = \frac{d\varsigma_g}{d\varsigma} .
$$
Observe that $J_\act$ is a homomorphism $J_\act : \G \to \R^+$, and the left Haar measure on $A \ract G$ is then given by $ds = J_\act(g)^{-1}d\varsigma(x)d\mu(g)$.

Two natural representations on $L^2(A)$ of $A$ and $\G$ are given respectively by the left regular representation of $A$ and the unitary representation associated with the action $\act$: for $x, x' \in A$, $g \in \G$ and $\psi \in L^2(A)$
\begin{equation}\label{eq:operators}
\lr_A(x)\psi(x') := \psi(x^{-1} x') \ , \quad T_\act(g)\psi(x) := J_\act(g)^{-\frac12} \psi(\act_{g^{-1}}(x)) .
\end{equation}

\begin{lem} \label{lem:Conmuting}
For the representations given in (\ref{eq:operators}) we have the following relation:
$$ T_\act (g) \lr_A(x) = \lr_A(\act_g(x)) T_\act (g)\,, \qquad g\in \G , x\in A\,.$$
\end{lem}

\begin{proof}
For $\psi \in L^2(A)$ and $x' \in A$ we have,
\begin{align*}
(T_\act (g) \lr_A(x) \psi)(x') & = J_\act (g)^{-\frac12} (\lr_A (x) \psi)(\act_{g^{-1}}(x')) \\
& = J_\act (g)^{-\frac12} \psi(x^{-1}\act_{g^{-1}}(x'))\,.
\end{align*}
On the other hand,
\begin{align*}
(\lr_A (\act_g(x)) T_\act (g) \psi)(x') & = (T_\act(g)\psi) (\act_g(x)^{-1} x') \\
& = J_\act (g)^{-\frac12} \psi(\act_{g^{-1}} (\act_g(x)^{-1} x'))\\
& = J_\act (g)^{-\frac12} \psi(x^{-1}\act_{g^{-1}}(x'))\,.\qedhere
\end{align*}

\end{proof}

\subsubsection{The quasiregular representation} \label{Subsec:quasiregular}

In terms of the representations (\ref{eq:operators}), one can define the so-called quasiregular representation $R(x,g) := \lr_A(x)T_\act(g) : S \to U(L^2(A))$, acting on $\psi \in L^2(A)$ as
\begin{equation}\label{eq:quasiregular}
R(x,g)\psi(x') = J_\act(g)^{-\frac12}\psi(\act_{g^{-1}}(x^{-1} x'))
\end{equation}
noting that, by Lemma \ref{lem:Conmuting}, $R(x,g)$ is a unitary representation of $S$
on $L^2(A).$

Consider $\widetilde \act$ the action of $S := A \ract \G$ on $A$ given by
\begin{equation}\label{eq:quairegularaction}
\widetilde \act_{(x,g)} (x') = x \act_g (x')\,, \qquad  x, x' \in A, \, g \in \G\,.
\end{equation}
It is easy to prove that $\widetilde \act$ is a measurable action of $S$ on the locally compact group $A$. Moreover, the unitary representation
$T_{\widetilde \act}$ of $S$ on $L^2(A)$, as given in (\ref{eq:repZak}), coincides with the quasiregular representation $R(x,g)$ defined in
(\ref{eq:quasiregular}). To prove this result choose $(x, g)\in S, x' \in A, \psi \in L^2(A)$ and apply the definition of $T_{\widetilde \act}$ to
write
$$
T_{\widetilde \act} (x,g) \psi (x') = J_\act (g)^{-\frac12} \psi(\widetilde\act_{(x,g)^{-1}}(x'))
$$
where we have used that $J_{\widetilde \act}(x,g) = J_\act(g)$.
The inverse of $(x,g)$ in the semidirect product $S$ is $(\act_{g^{-1}}(x^{-1}), g^{-1})$, and hence $\widetilde \act_{(x,g)^{-1}}(x') =
\act_{g^{-1}}(x^{-1})\, \act_{g^{-1}}(x') = \act_{g^{-1}}(x^{-1}x'). $ This shows the result, since
\begin{equation}\label{eq:quaasireg}
T_{\widetilde \act} (x,g) \psi (x') = J_\act (g)^{-\frac12} \psi(\act_{g^{-1}}(x^{-1}x'))= R(x,g)\psi (x')\,.
\end{equation}
In order to be able to apply Theorem B, we choose $B$ a discrete countable subgroup of $A$ and $\Gamma$ a discrete
countable subgroup of $\G$ in such a way that the action $\act$ when restricted to $\Gamma$ is an action on $B$. Then $\Sigma := B \ract \Gamma$
is a subgroup of $S = A \ract \G$. Assume that $\widetilde \act : \Sigma \to \textnormal{Aut}(A)$ given by $\widetilde \act_{(b,\gamma)}(y) =
b \,\act_\gamma(y)$ has the tiling property on $A$. That is, there exists a measurable set $C \subset A$ such that
$$
\Big\{ \widetilde\act_{(b,\gamma)} (C)\Big\}_{(b,\gamma)\in \Sigma} = \Big\{b\,\act_\gamma (C)\Big\}_{b\in B, \gamma \in \Gamma}
$$
is an almost everywhere covering partition of $A.$ In this situation, the quasiregular representation
$$
R(b,\gamma) \psi (x) = J_\act (\gamma)^{-\frac12} \psi (\act_{\gamma^{-1}}(b^{-1}x))
$$
is dual integrable with bracket
$$
[\psi_1 , \psi_2] = \int_C Z_R[\psi_1] (x)\, Z_R[\psi_2] (x)^*\,d\varsigma(x)\,, \quad  \psi_1, \psi_2 \in L^2(A)\,,
$$
where
$$
Z_R[\psi](x) = \sum_{\gamma \in \Gamma} \sum_{b\in B} R(b, \gamma) \psi(x)\, \lr(b,\gamma)^*
$$
is the Zak transform from $L^2(A, \varsigma)$ to $L^2(C, L^2(\vn(B \ract \Gamma))$.

\paragraph{Euclidean motion groups $E(n) = \R^n \ract O(n)$}\ \\
\noindent
For $x \in \R^n = A$ and $r \in O(n) = \G$ we can define the composition law on $\R^n \times O(n)$ by
$$
(x,r)\bullet(x',r') = (x + rx', rr') .
$$
This is a semidirect product with $\act: O(n) \to \textnormal{Aut}(\R^n)$ given by $\act(r)(x) = rx := \act_r (x)$, and $J_\act = 1$. Let $B = L \Z^n$
be a full rank lattice of $\R^n$, that is $L$ is a non-singular $n\times n$ matrix,
and $\Gamma$ be a finite subgroup of $O(n)$ such that $\act_\gamma(B) = B$ for all $\gamma \in \Gamma$.
Groups of the type $L\Z^n \ract \Gamma$ are called crystallographic groups.
It is proved in  \cite{Manning12} (Proposition 5) that there exists a compact set $C \subset \R^n$ such that $\{b+rC\}_{b\in B, r\in \Gamma}$ is
an almost everywhere partition of $\R^n$, in such a way that
$$\Big\{b + \bigcup_{r\in \Gamma} rC\Big\}_{b \in B}$$
is also an almost everywhere partition of $\R^n$.

Our Theorem B gives that the quasiregular representation $R(b,r)$
of the semidirect product $L\Z^n \ract \Gamma$ on $L^2(\R^n)$ given by
$$
R(b,r)\psi (x) = \psi(r^{-1}(x-b))\, \qquad  x\in \R^n,\, b\in B, \, r\in \Gamma,
$$
is dual integrable, and the bracket is given in terms of the Zak transform
$$
Z_R[\psi](x) = \sum_{r\in \Gamma} \sum_{b\in L\Z^n} \psi(r^{-1} (x-b)) \lr(b, r)^*\,.
$$
A different approach for crystallographic groups, leading to an alternative notion of bracket map, but also operator valued, has
been developed in \cite{Manning12}.

\

As a simple, but instructive, example consider $B = \Z^2 \subset \R^2$ and $\Gamma$ the group of symmetries of a square. The group $\Gamma$
is generated by the rotation $r$ and the reflection $s$ given in matrix notation by
$$
r=\left(\begin{array}{cc}
    0 & -1\\
    1 & 0
\end{array}\right), \qquad
s=\left(\begin{array}{cc}
    0 & 1\\
    1 & 0
\end{array}\right).
$$
Therefore, $\Gamma = \{I, r, r^2, r^3, s, rs, r^2s, r^3s\}$ with $r^4=I, s^2=I,$ and $rs=sr^3.$ In this case a compact set $C \subset \R^2$
is
$$
C =\{ (x,y)\in \R^2 : 0 \leq x \leq \frac12\,, 0 \leq y \leq x\}\,.
$$
Observe that $\{\gamma C\}_{\gamma \in \Gamma}$ tiles the square $K = [-1/2, 1/2]^2$, and $K$ is a tiling set of $\R^2$ by integer translations 
(see Figure 1).

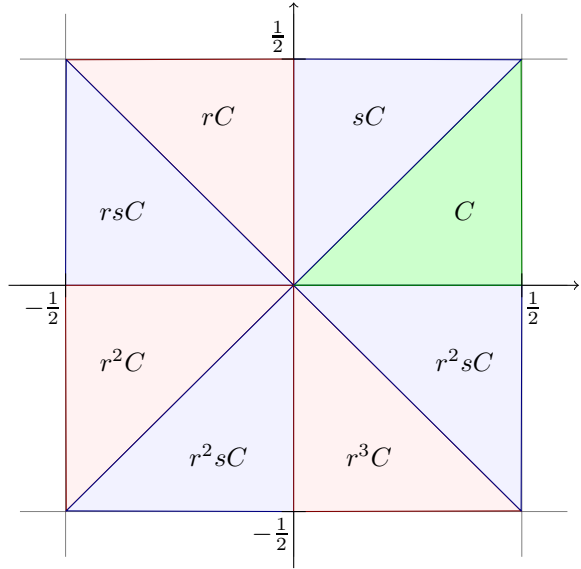
\begin{figure} \label{figure1}
\centering
\begin{tikzpicture}[scale=3]
 \draw[step=1cm, gray, very thin] (-1.2,-1.2) grid (1.2,1.2);
 \draw[->] (-1.25,0) -- (1.25,0) coordinate (x axis);
 \draw[->] (0,-1.25) -- (0,1.25) coordinate (y axis);

 \draw[green!50!black, fill=green!20] (0,0) -- (1,0) -- (45:1.41cm) -- (0,0);
 \node at (0.75,0.33) {$C$};

 \draw[blue!50!black, fill=blue!5] (0,0) -- (0,1) -- (45:1.41cm) -- (0,0);
 \node at (0.33,0.75) {$sC$};

 \draw[red!50!black, fill=red!5] (0,0) -- (0,1) -- (135:1.41cm) -- (0,0);
 \node at (-0.33,0.75) {$rC$};

 \draw[blue!50!black, fill=blue!5] (0,0) -- (-1,0) -- (135:1.41cm) -- (0,0);
 \node at (-0.75,0.33) {$rsC$};

 \draw[red!50!black, fill=red!5] (0,0) -- (-1,0) -- (225:1.41cm) -- (0,0);
 \node at (-0.75,-0.33) {$r^2C$};

 \draw[blue!50!black, fill=blue!5] (0,0) -- (0,-1) -- (225:1.41cm) -- (0,0);
 \node at (-0.33,-0.75) {$r^2sC$};

 \draw[red!50!black, fill=red!5] (0,0) -- (0,-1) -- (-45:1.41cm) -- (0,0);
 \node at (0.33,-0.75) {$r^3C$};

 \draw[blue!50!black, fill=blue!5] (0,0) -- (1,0) -- (-45:1.41cm) -- (0,0);
 \node at (0.75,-0.33) {$r^2sC$};

\draw[green!50!black] (0,0) -- (1,0);

 \draw (-1cm,1.5pt) -- (-1cm,-1.5pt);
 \node at (-1.1,-0.1) {$-\frac12$};
 \draw (1cm,1.5pt) -- (1cm,-1.5pt);
 \node at (1.05,-0.1) {$\frac12$};
 \draw (1.5pt,-1cm) -- (-1.5pt,-1cm);
 \node at (-0.1,-1.1) {$-\frac12$};
 \draw (1.5pt,1cm) -- (-1.5pt,1cm);
 \node at (-0.07,1.1) {$\frac12$};
\end{tikzpicture}
\caption{Tiling with rotations and reflections.}
\end{figure}

\paragraph{Heisenberg motion groups $M(n) = \Heis_n \ract U(n)$}\ \\
\noindent
The Heisenberg group $\Heis_n$ is the group $(\C^n \times \R, \circ)$ with composition law
$$
(z,t)\circ(z',t') = (z + z', t + t' - \frac12 \Im(z'.z))
$$
where $z.w$ stands for the product in $\C^n$. For $(z,t) \in \C^n \times \R$ and $u \in U(n)$ we can then define the composition law on $\Heis_n \times U(n)$
$$
(z,t,u)\bullet(z',t',u') = (z + uz', t + t' - \frac12 \Im(uz'.z), uu') .
$$
This is a semidirect product with $\act: U(n) \to \textnormal{Aut}(\Heis^n)$ given by $\act_u (z,t) = (uz,t)$. The previous analysis can be applied in this case considering a discrete subgroup $B$ of $\Heis_n$ such as $B = (\Z^{2n} \times \frac12 \Z, \circ)$ and $\Gamma$ a discrete subgroup of $U(n)$ such that $\sigma_\gamma$ are automorphisms of $B$ for $\gamma \in \Gamma$.

\subsubsection{The affine representation} \label{Subsec:affine-representation}
Given two locally compact groups $A$ and $\G$, one can construct another representation, that we will call the \emph{affine} representation, which is used in the theory of composite wavelets (see e.g. \cite{WeissWilson01, LaugersenWeaverWeissWilson02, LabateWeiss10}). For $\lr_A$ and $T_\act$ as in (\ref{eq:operators}), such representation is given by $\Pi(x,g):= T_\act(g)\lr_A(x)$ and acts on $\psi \in L^2(A)$ as
\begin{equation}\label{eq:affine}
\Pi(x,g)\psi(x') = J_\act(g)^{-\frac12}\psi(x^{-1}\act_{g^{-1}}(x')).
\end{equation}
Using Lemma \ref{lem:Conmuting} we deduce
\begin{equation}\label{eq:affine-representation}
\Pi(x,g)\, \Pi(x',g') = \Pi (\act_{g'^{-1}}(x)\, x' , gg').
\end{equation}
An alternative way to prove (\ref{eq:affine-representation}) is to observe that $\Pi(x,g) = R(\act_g(x), g)$ and use that the quasiregular
representation $R$ is a representation of $S = A \ract \G$ on $L^2(A).$ We note, in passing, that $\Pi(x,g) = R(x^{-1}, g^{-1})^*$, whose proof is left to the reader.

Therefore, $\Pi$ is the representation on $L^2(A)$ of a group $W$ which coincides with $A \times \G$ as a topological space,
and is endowed with the composition law
\begin{equation}\label{eq:affinecomp}
(x,g)\odot (x',g') = (\act_{g'^{-1}}(x)\,x', gg') .
\end{equation}
We note that even if the law (\ref{eq:affinecomp}) is different from the law (\ref{eq:semidirectcomp}), there is a
topological group homomorphism $h : (W,\odot) \to (S,\bullet)$ given by
$$
h(x,g) = (\act_{g}(x), g) .
$$
Indeed
\begin{align*}
h(x,g)&\bullet h(x',g') = (\act_{g}(x), g) \bullet (\act_{g'}(x'), g') = (\act_{g}(x)\,\act_{g}(\act_{g'}(x')), gg')\\
& = (\act_{gg'}( (\act_{g'^{-1}}(x) \,x'), gg') = h(\act_{g'^{-1}}(x)\,x', gg') = h\big((x,g)\odot (x',g')\big) .
\end{align*}
Consider now the action $\widetilde \act$ of $W$ on $A$ given by
\begin{equation} \label{eq:action-new}
\widetilde \act_{(x,g)} (y) = \act_g (xy)\,, \qquad  x, y \in A, \, g\in \G\,.
\end{equation}
Since the inverse of $(x,g)$ with the operation $\odot$ is $(\act_g(x^{-1}), g^{-1}),$ the representation $T_{\widetilde \act}$
associated to $\widetilde \act$ as in (\ref{eq:repZak}) reads
\begin{align*}
(T_{\widetilde \act} (x,g) \psi)(y)  & = J_\act(g)^{-\frac12} \psi(\widetilde \act_{(x,g)^{-1}}(y)) =
     J_\act(g)^{-\frac12} \psi(\act_{g^{-1}} (\act_g(x^{-1})\,y) \\
& = J_\act(g)^{-\frac12}\psi(x^{-1}\act_{g^{-1}}(y)) = \Pi(x,g)\psi(y).
\end{align*}
Therefore, $\Pi(x,g)$ coincides with the unitary representation $T_{\widetilde \act}$ (as in \ref{eq:repZak}) of the group
$W$ on $L^2(A)$ induced by the action $\widetilde \act$ given by (\ref{eq:action-new}).

Choose now $B$ a discrete countable subgroup of $A$ and $\Gamma$ a discrete countable subgroup of $\G$, in such a way that
the action $\act$ when restricted to $\Gamma$ is an action on $B$. Then $B\times \Gamma$, endowed with the composition law
$\odot$ given in (\ref{eq:affinecomp}), is a subgroup of $(W, \odot).$  If we assume that the action $\widetilde \act$ given in
(\ref{eq:action-new}) has the tiling property we could apply Theorem B to deduce that $\Pi(x,g)$ is dual integrable
and the bracket map is given in terms of the Zak transform, which in this case reads
$$
Z_\Pi [\psi](x) = \sum_{b \in B}\sum_{\gamma \in \Gamma} J_\act (\gamma)^{-\frac12} \psi (b^{-1} \act_{\gamma^{-1}}(x)) \lr(b,\gamma)^*\,.
$$
The tiling condition is satisfy if there exist a measurable set $C \subset A$ such that $\{\widetilde \act_{(b,\gamma)} (C)\}_{b\in B, \gamma\in \Gamma}$
is an almost everywhere partition of $A$. Since
$$
\widetilde \act_{(b,\gamma)}(C) = \act_\gamma (b\,C) = \act_\gamma(b)\,\act_\gamma(C)
$$
and $\act_\gamma \in \textnormal{Aut}(B),$ this is equivalent to require that $\{b\,\act_\gamma (C)\}_{(b,\gamma) \in B \times \Gamma}$
is an almost everywhere partition of $A$, as in the case of the quasiregular representation.

%

\vspace{20pt}
\noindent
\textsc{Davide Barbieri}:
Departamento de Matem\'aticas, Universidad Aut\'onoma de Madrid, 28049 Madrid, Spain, \emph{davide.barbieri@uam.es}

\vspace{4pt}
\noindent
\textsc{Eugenio Hern\'andez}:
Departamento de Matem\'aticas, Universidad Aut\'onoma de Madrid, 28049 Madrid, Spain, \emph{eugenio.hernandez@uam.es}

\vspace{4pt}
\noindent
\textsc{Javier Parcet}:
Instituto de Ciencias Matem\'aticas, CSIC-UAM-UC3M-UCM, 28049 Madrid, Spain, \emph{javier.parcet@icmat.es}

\end{document}